\documentclass[11pt]{article}
\usepackage{hyperref}
\usepackage{amsthm,amssymb}
\usepackage{amsmath}
\usepackage{color}
\usepackage{authblk}
\usepackage{thm-restate}
\usepackage{enumerate}
\usepackage{tikz}

\newtheorem{THM}{Theorem}[section]
\newtheorem{LEM}[THM]{Lemma}
\newtheorem{COR}[THM]{Corollary}
\newtheorem{PROP}[THM]{Proposition}

\newtheorem{QUE}[THM]{Question}

\newtheorem{OBS}[THM]{Observation}

\newcommand\abs[1]{\lvert #1\rvert}
\newcommand{\C}{\mathcal{C}}

\newcommand{\CA}{\mathcal{A}}

\newcommand{\Fan}{\operatorname{FAN}}
\newcommand{\fan}{\operatorname{Fan}}
\newcommand{\mc}[1]{\mathcal{#1}}

\newcommand{\emb}{\hookrightarrow}

\title{Obstructions for partitioning into forests and outerplanar graphs}
\date\today

\author[1]{Ringi Kim\thanks{ Supported by the National Research Foundation of Korea (NRF) grant
funded by the Korea government (MSIT) (NRF-2018R1C1B6003786), and 
also 
by Basic Science Research Program through the National Research Foundation of Korea (NRF) funded by the Ministry of Education (2019R1A6A1A10073887).
}}%
\author[2]{Sergey Norin\thanks{Supported by an NSERC discovery grant}}
\author[3,1]{Sang-il Oum\thanks{Supported by the Institute for Basic Science (IBS-R029-C1).}}

\affil[1]{Department of Mathematical Sciences, KAIST,  Daejeon, Korea.}
\affil[2]{Department of Mathematics and Statistics, McGill University, Montr\'eal, Canada.}
\affil[3]{Discrete Mathematics Group, Institute for Basic Science (IBS), Daejeon, Korea.}
\affil[ ]{\small kimrg@kaist.ac.kr, snorine@gmail.com, sangil@ibs.re.kr}

\begin{document}
\maketitle
\begin{abstract}
  For a class $\mathcal C$ of graphs,
  we define  \emph{$\mathcal C$-edge-brittleness} of a graph $G$
  as the minimum $\ell$ such that
  the vertex set of $G$ can be partitioned into sets inducing a subgraph in $\mathcal C$
  and there are $\ell$ edges having ends in distinct parts.
  We characterize classes of graphs having bounded $\mathcal C$-edge-brittleness
  for a class $\C$ of forests or a class $\C$ of graphs with no $K_4\setminus e$ topological minors
  in terms of forbidden obstructions.
  We also define \emph{$\mathcal C$-vertex-brittleness} of a graph $G$
  as the minimum $\ell$ such that
  the edge set of $G$ can be partitioned into sets inducing a subgraph in $\mathcal C$
  and there are $\ell$ vertices incident with edges in distinct parts.
  We characterize classes of graphs having bounded $\mathcal C$-vertex-brittleness
  for a class $\C$ of forests
  or a class $\C$ of outerplanar graphs
  in terms of forbidden obstructions.
  We also investigate the relations between the new parameters and the edit distance.
\end{abstract}

\section{Introduction}

How far is a graph $G$ from a graph class $\C$? 

A natural and well-studied (see e.g.~\cite{AlonStav08,Martin16}) measure of this distance is defined as follows.
The \emph{edit distance} $e_{\C}(G)$ from a graph $G$ to a graph class $\C$ is the minimum of $|E(G') \triangle E(G)|$ taken over all $G' \in \C$ with $V(G')=V(G)$. If $\C$ is  \emph{monotone}, that is, $\C$ is closed under isomorphisms and taking subgraphs, then  $e_{\C}(G)$ is simply the minimum number of edges one needs to delete from $G$ to obtain a graph in $\C$. In this paper, we only consider simple graphs and monotone graph classes.

In addition to the edit distance we consider two alternative distance parameters, which measure how difficult it is to partition a graph $G$ into parts which belong to the class $\C$.  For a  graph $G$, the \emph{$\C$-edge-brittleness} of $G$, denoted by $\eta_{\C}(G)$,
is the minimum integer $\ell$ such that there is a partition $(V_1,V_2,\ldots,V_n)$ of $V(G)$
such that $G[V_i]\in \C$ for all $i$
and the number of edges having ends in distinct $V_i$'s is $\ell$.\footnote{Given a graph $G$ and $X \subseteq V(G)$ we denote by $G[X]$ the subgraph of $G$ induced by~$X$.} 
A dual parameter, the \emph{$\C$-vertex-brittleness} of $G$,  denoted by $\kappa_{\C}(G)$,
is the minimum integer $\ell$ such that
there is a partition $(E_1,E_2,\ldots,E_n)$ of $E(G)$
with the property that the subgraph of $G$ induced by the edges in $E_i$ belongs to $\C$ for each $i$
and the number of vertices incident with edges in distinct $E_i$'s is $\ell$.

Finally, we add yet another parameter to our list, which comes from a packing problem dual to the covering problem which defines the edit distance.  For a  graph $G$, the \emph{$\bar{\C}$-capacity} of $G$, denoted by $\nu_{\C}(G)$, is the maximum integer $\ell$ such that there exist edge-disjoint subgraphs $H_1,H_2,\ldots,H_{\ell}$ of $G$ such that $H_i$ does not belong to $\C$ for each $i$. 

The following easy observation describes  the basic relations  between the edit distance and the above parameters. 

\begin{OBS}\label{obs:basic}	Let $\C$ be a monotone graph class such that $K_1,K_2 \in \C$,  and $\C$ is closed under taking disjoint unions.\footnote{We write $K_n$ and $K_{m,n}$ for the complete graph on $n$ vertices
	and the complete bipartite graph on $m+n$ vertices partitioned into sets of $m$ and $n$ vertices. } Then for every graph $G$ we have $e_{\C}(G) \leq \eta_{\C}(G)$, $\kappa_{\C}(G)/2 \leq e_{\C}(G)$, and $\nu_{\C}(G) \leq e_{\C}(G)$.
\end{OBS}

\begin{proof}
	Let $(V_1,V_2,\ldots,V_n)$ be a partition of $V(G)$ such that $G[V_i]\in \C$ for all $i$ and $\eta_{\C}(G)$ edges of $G$ have ends in distinct parts of this partition. Let $G' = \bigcup_{i=1}^n G[V_i]$. Then $G' \in \C$, as $\C$ is closed under taking disjoint unions, and $G'$ is obtained from $G$ by deleting  $\eta_{\C}(G)$ edges, implying the first inequality.
	
	Let $F$ be a set of edges such that $|F|= e_{\C}(G)$ and $G\setminus F \in \C$. For the second inequality, consider a partition of $E(G)$ into $E(G)-F$ and $|F|$ parts of size one corresponding to elements of $F$. The subgraph of $G$ induced by each part of this partition lies in $\C$ by our assumptions, while the number of vertices incident with two edges in two distinct parts is clearly at most $2|F|=2 e_{\C}(G)$. It follows that  $\kappa_{\C}(G) \leq 2e_{\C}(G)$, as desired.
	
	Finally, let $\ell = \nu_{\C}(G)$, and let $H_1,H_2,\ldots,H_{\ell}$ be as in the definition of $\nu_{\C}(G)$. 
	Let $F$ be as in the previous paragraph. Then, $F \cap E(H_i) \neq \emptyset$ for every $i$. Thus $|F| \geq \ell$, implying the last inequality.
\end{proof}
Recall that for a graph $G$, a \emph{subdivision} of $G$ is a graph obtained from $G$ by replacing edges of $G$ with internally disjoint paths of length at least~$1$.
For graphs $G$ and $H$, we say $H$ is a \emph{topological minor} of $G$ if
$G$ has a subgraph that is a subdivision of $H$.
A graph $H$ is a \emph{minor} of a graph $G$ if $H$ can be obtained from $G$
by a sequence of edge contractions, edge deletions, and vertex deletions.
A graph is \emph{outerplanar} if it can be drawn on the plane without edge crossings such that there is a face incident with all vertices.
We say $G$ is \emph{$H$-free}, if
no topological minor of $G$ is isomorphic to $H$.
For a set $\mathcal{H}$ of graphs, we say $G$ is $\mathcal{H}$-free if $G$ is $H$-free for every $H \in \mathcal{H}$.
A \emph{diamond} is the graph obtained from $K_4$ by removing one edge, see Figure~\ref{fig:diamond}.
\begin{figure}
	\centering
	\tikzstyle{v}=[circle, draw, solid, fill=black, inner sep=0pt, minimum width=3pt]
	\begin{tikzpicture}
	\node[v] at (0,0) (v1) {};
	\node[v] at (60:.5) (v2) {};
	\node[v] at (-60:.5) (v4) {};
	\node[v] at (.5,0) (v3) {};
	\draw (v1)--(v2)--(v3)--(v4)--(v1);
	\draw (v1)--(v3);
	\end{tikzpicture}

	\caption{The diamond graph $D$.}
	\label{fig:diamond}
\end{figure}
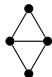

The goal of this paper is to investigate structural reasons that guarantee that a graph is far from a given class $\C$, using each of the above measures of distance. In other words,  we attempt to determine qualitative obstructions to partitioning a graph into graphs in $\C$. 
We concentrate on the class $\CA$ of forests, the class $\mc{O}$ of outerplanar graphs, and the class $\mc{D}$ of diamond-free graphs, that is intermediate between $\CA$ and $\mc{O}$.

These classes are not only monotone, but also closed under topological minors.
For this paper, we say that a class $\C$ of graphs is an \emph{ideal} if $\C$ is closed under isomorphisms and taking topological minors. As will be shown later (Proposition~\ref{prop:topteb}), if $\C$ is an ideal, then taking topological minors does not increase edit distance to $\C$, $\C$-vertex-brittleness, or $\bar{\C}$-capacity. Thus a characterization of  minimal ideals which have unbounded  edit distance to $\C$  (respectively,  $\C$-vertex-brittleness, and $\bar{\C}$-capacity) is an explicit and convenient description of the obstructions we are interested in. Moreover, we will show that $\C$-edge-brittleness is bounded by a function of the edit distance to $\C$ for the classes that we consider (although not in general). Thus we will also obtain a description of obstructions for unbounded $\C$-edge-brittleness.

Our first two theorems give a  characterization of  minimal ideals which have unbounded distance to the class of forests. 

Let $G$ be a graph, $S \subsetneq V(G)$, and let $k$ be a positive integer. A graph $\fan(G,S,k)$ is obtained from $k$ vertex-disjoint copies of $G$ by  identifying all the copies of $v$ for each $v \in S$ 
and identifying all the copies of $e$ for each $e \in E(G[S])$.
Thus, for example, $\fan(G, \emptyset,k)$ is the disjoint union of $k$ copies of $G$, and $\fan(K_2,\{v\},k)$ for $v \in V(K_2)$ is a star with $k$ leaves. Let $\Fan(G,S)$ denote the ideal 
consisting of all graphs isomorphic to a topological minor of $\fan(G,S,k)$ for some positive integer $k$. 
We say that a graph parameter $f$ is \emph{bounded} on a class $\mathcal G$ of graphs if there is a constant~$M$ such that $\abs{f(G)}\le M$ for all $G\in \mathcal G$.
\begin{restatable}{THM}{forestedge}
\label{thm:forest_edge}
  Let $\mathcal{G}$ be an ideal.
Then, the following are equivalent.
\begin{enumerate}
	\item $\nu_{\CA}$ is bounded on $\mc{G}$.
	\item $e_{\CA}$ is bounded on $\mc{G}$.
	\item $\eta_{\CA}$ is bounded on $\mc{G}$.
	\item $\Fan(K_3,S) \not\subseteq \mathcal G$ for every $S \subseteq V(K_3), |S| \leq 2$.
\end{enumerate}
\end{restatable}

Thus the parameters $\nu_{\CA}$,
$e_{\CA}$, and $\eta_{\CA}$ are tied. On the other hand, it is easy to see that $\kappa_{\CA}(K_{2,n})=2$, while $\nu_{\CA}(K_{2,n}) \geq \lfloor n/2\rfloor$. 
The next theorem characterizes the graph classes with bounded $\C$-vertex-brittleness, showing that this example is essentially the only source of discrepancy.
 
\begin{restatable}{THM}{forestvertex}\label{thm:forest_vertex}
 	Let $\mathcal{G}$ be an ideal.
 	Then,  $\kappa_{\CA}$ is bounded on $\mc{G}$ if and only if
 	$\Fan(K_3,S) \not\subseteq \mathcal G$ for every $S \subseteq V(K_3), |S| \leq 1$.
 \end{restatable}

Next we extend Theorem~\ref{thm:forest_edge} to the class of diamond-free graphs.  
Let $D$ be the diamond graph in Figure~\ref{fig:diamond}. 
Let $\mc{D}$ denote the class of $D$-free graphs. 
For convenience let us define another special ideal $\mc{K}_{2,*}$ consisting of graphs which are isomorphic to a topological minor of $K_{2,n}$ for some $n$. It is easy to see that 
$\mc{K}_{2,*}=\Fan(K_3,S)$ for $S \subseteq V(K_3)$ with $|S|=2$.

\begin{restatable}{THM}{diamondedge}\label{thm:diamond_edge}
  Let $\mathcal{G}$ be an ideal. Then, the following are equivalent.
  \begin{enumerate}
  	\item $\nu_{\mc{D}}$ is bounded on $\mc{G}$.
  	\item $e_{\mc{D}}$ is bounded on $\mc{G}$.
  	\item $\eta_{\mc{D}}$ is bounded on $\mc{G}$.
  	\item $\Fan(D,S) \not\subseteq \mathcal G$ for every $S \subseteq V(D)$ with $|S| \leq 1$, and
  	$\mc{K}_{2,*} \not\subseteq \mathcal G$.
  \end{enumerate}
\end{restatable}

The analogue of Theorem~\ref{thm:forest_vertex} also holds.

\begin{restatable}{THM}{diamondvertex}\label{thm:diamond_vertex}
	Let $\mathcal{G}$ be an ideal.
	Then,  $\kappa_{\mc{D}}$ is bounded on $\mc{G}$ if and only if
	$\Fan(D,S) \not\subseteq \mathcal G$ for every $S \subseteq V(D)$ such that $|S| \leq 1$, or $S$ consists of the pair of degree-$2$ vertices of $D$.
\end{restatable}

We were unable to obtain an analogue of Theorems~\ref{thm:forest_edge} and~\ref{thm:diamond_edge} for outerplanar graphs. The difficulty partially stems from the fact  that $\eta_{\mc{O}}$ is no longer tied to the other parameters, as the next proposition shows.

\begin{restatable}{PROP}{example}\label{prop:example}
For every integer $\ell >0 $ there exists a graph $G=G(\ell)$ such that $e_{\mc{O}}(G)=1$ and $\eta_{\mc{O}}(G) \geq \ell$.  
\end{restatable} 

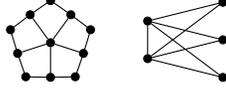
\begin{figure}
  \centering
  \tikzstyle{v}=[circle, draw, solid, fill=black, inner sep=0pt, minimum width=3pt]
  \begin{tikzpicture}[scale=.4]
    \draw(0,0) node[v](x1){};
        \foreach \i in {0,3,4}{
          \draw (\i*72+90:1.4) node[v] (v\i){};
          \foreach \j in {1}{
            \draw  (\i*72+90+46-\j*10:.15*\j+1) node[v] (v\i-\j){};
            \draw (x\j)--(v\i-\j);
            \draw plot [smooth] coordinates { (v\i) (v\i-\j) (\i*72+90+72:1.4)};
          }
        }
        \foreach \i in {1,2}{
          \draw (\i*72+90:1.4) node[v] (v\i){};
          \foreach \j in {1}{
            \draw  (\i*72+90+26+\j*10:.15*\j+1) node[v] (v\i-\j){};
            \draw (x\j)--(v\i-\j);
            \draw plot [smooth] coordinates { (v\i) (v\i-\j) (\i*72+90+72:1.4)};
          }
        }
      \end{tikzpicture}
      $\quad$
   \begin{tikzpicture}[yscale=2]
     \node[v] at (0,.375) (x){};
     \node[v] at (0,.125) (y){};
     \foreach \y in {0,0.25,.5}{
       \draw (x)--(1,\y) node[v]{} -- (y);
     }
     \draw (x)--(y);
     \end{tikzpicture}
  
  \caption{Two graphs $W_5^+$ and $K_{2,3}^+$.}
  \label{fig:wk}
\end{figure}
We were, however, able to characterize the minimal ideals with unbounded $\mc{O}$-vertex-brittleness. For an integer $k \geq 3$, let $W^{+}_k$ denote the graph obtained from the  wheel on $k+1$ vertices by subdividing every edge of the rim.  Equivalently, $W^{+}_k$ is obtained from a cycle on $2k$ vertices by adding an extra vertex  adjacent to all the vertices of some independent set of size $k$ in the cycle. Let $K^{+}_{2,3}$ denote the graph obtained from $K_{2,3}$ by adding an edge joining the degree-$3$ vertices, see Figure~\ref{fig:wk}.

\begin{restatable}{THM}{outervertex}\label{thm:outer_vertex}
  Let $\mathcal{G}$ be an ideal.
  Then, $\mathcal{G}$ has bounded $\kappa_{\mc{O}}$ if and only if
  it contains none of the following ideals.
  \begin{itemize}
  \item   $\Fan(K_4,S)$ with $S \subseteq V(K_4)$, $|S| \leq 1$.
  \item $\Fan(K_{2,3},S)$ with $S \subseteq V(K_{2,3})$ such that $|S| \leq 1$ or $S$ consists of two degree-$2$ vertices.
   \item $\Fan(K^{+}_{2,3},S)$ with $S \subseteq V(K_{2,3}^+)$ consisting of all degree-$2$ vertices.
  \item  $\Fan(W^{+}_k,S)$ where $k \geq 3$, and $S \subseteq V(W^{+}_k)$ is the set of all degree-$2$ vertices.
  \end{itemize}
\end{restatable}

We finish this section with the proof that taking topological minors does not increase our measures of distance to an ideal, except possibly for edge-brittleness.

It is convenient to present our proof using the language of embeddings. It is easy to see that
a subdivision of a graph $G'$ is isomorphic to a subgraph of a graph  $G$ if and only if there exists a map $\phi$ defined on $V(G') \cup E(G')$ such that  $\phi$ maps $V(G')$ injectively into $V(G)$, and $\phi$ maps the edges of $G'$ into internally disjoint paths, so that $\phi(uv)$ has ends $\phi(u)$ and $\phi(v)$ for every $uv \in E(G')$. We refer to a map with these properties as an \emph{embedding of $G'$ into $G$}, and write $\phi: G' \emb G$ to denote that $\phi$ is such an embedding. 
For a subgraph $F$ of $G'$, 
let $\phi(F)$ denote the subgraph of $G$ such that $V(\phi(F)) = \{ \phi(v) : v \in V(F)\} \cup (\bigcup_{e \in  E(F)}V(\phi(e)))$ and  $E(\phi(F))  = \bigcup_{e \in  E(F)}E(\phi(e))$. Note that $\phi(F)$ is isomorphic to a subdivision of $F$. 

\begin{PROP}\label{prop:topteb}
Let $\C$ be an ideal.
If $G'$ is a topological minor of a graph $G$, then 
\[e_{\C}(G') \le e_{\C}(G),  \kappa_{\C}(G') \le \kappa_{\C}(G),   \text{ and } \nu_{\C}(G') \le \nu_{\C}(G) .\]
\end{PROP}
\begin{proof}
	
	Let $\phi:G' \emb G$ be an embedding. Let $F \subseteq E(G)$ with $|F|=e_{\C}(G)$ be such that $G \setminus F \in \C$. Let $F' = \{f \in E(G') : E(\phi(f)) \cap F \neq \emptyset.\}$. Then there is a restriction of $\phi$ that is an embedding of $G' \setminus F'$ into $G \setminus F$. It follows that $G' \setminus F' \in \C$, and so $e_{\C}(G') \leq |F'| \leq |F| =e_{\C}(G)$. This proves the first inequality.

For the second inequality, let $\mc{E} = (E_1,\ldots, E_k)$ be a partition of $E(G)$ so that the set $W$ of all vertices incident  with edges in distinct parts of $\mc{E}$ satisfies  $|W| = \kappa_\C(G)$. Let $v_1,v_2, \ldots,v_n$ be an arbitrary ordering of $V(G')$, and let  $\mc{E}'=(E'_1,\ldots, E'_k)$ be a partition of $E(G')$ such that for $e'= v_{i}v_{i'} \in E(G')$ with $i < i'$ we have $e' \in E'_s$, where $s$ is chosen so that the unique edge $e$ of $\phi(e')$ incident with $\phi(v_i)$ satisfies $e \in E_s$. Let $W'$ be the set of all vertices of $v_i \in V(G')$ such that $\phi(v_i) \in W$ or there exists $i' < i$ so that some internal vertex of $\phi(v_iv_{i'})$ lies in $W$. It is easy to see that $|W'| \leq |W|$ and any vertex of $G'$ incident   with edges in distinct parts of $\mc{E}'$ lies in $W'$. Thus $\kappa_\C(G') \leq |W'| \leq |W| = \kappa_\C(G)$. This proves the second inequality.

Finally, let $H_1,\ldots,H_{\ell}$ be pairwise edge-disjoint subgraphs of $G'$ such that $H_i \not \in \mc{C}$ and $\ell = \nu_{\C}(G')$. Then  $\phi(H_1),\ldots,\phi(H_\ell)$ are edge-disjoint subgraphs of $G$, and we have $\phi(H_i) \not \in \C$ as $\C$ is an ideal. It follows that $\nu_{\C}(G) \geq \ell  = \nu_{\C}(G')$.
\end{proof}

The bounds in Proposition~\ref{prop:topteb} do not necessarily hold
if $G'$ is a minor of~$G$.
For example, let $G$ be the graph given in Figure~\ref{fig:ex}
and let $G'=G/e$ for the edge $e$ shown in the figure.
The partition of $E(G)$ into three internally disjoint paths 
connecting the degree-$3$ vertices shows that $\kappa_{\mc{A}}(G)=2$.
It is  easy to see that $\kappa_{\mc{A}}(G/e)=3$, $\nu_{\mc{A}}(G)=1$, and $\nu_{\mc{A}}(G/e)=2$. 
The inequality $\eta_{\C}(G') \le \eta_{\C}(G)$ does not necessarily hold even if $G'$ is a topological minor of a graph $G$. 
For example, let $\C$ be the class of all $K_n$-free graphs, let $G'=K_n$  and let $G$ be obtained by replacing one edge of $K_n$ by a path of length two. 
It is easy to see in this case that $\eta_{\C}(G')=n-1$, $\eta_{\C}(G)= 2$.
Similarly an example in Figure~\ref{fig:ex2} shows that $e_{\mathcal O}$ may increase by contracting an edge.

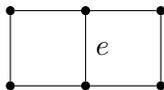
\begin{figure}
  \centering
 \tikzstyle{v}=[circle, draw, solid, fill=black, inner sep=0pt, minimum width=3pt]
  \begin{tikzpicture}
    \foreach \x in {1,2,3}{
        \node[v] (v\x0) at(\x,0)  {};
      }
    \foreach \x in {1,2,3}{
        \node[v] (v\x1) at(\x,1)  {};
      }
    \draw (v10)--(v11)--(v31)--(v30)--(v10);
    \draw (v20)  -- node [right]{$e$} (v21);
  \end{tikzpicture}
  \caption{An example showing that $\kappa_{\mathcal A}$ and $\nu_{\mathcal A}$ may increase by contracting the edge $e$.}
  \label{fig:ex}
\end{figure}
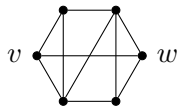
\begin{figure}
  \centering
 \tikzstyle{v}=[circle, draw, solid, fill=black, inner sep=0pt, minimum width=3pt]
 \begin{tikzpicture}[scale=.7]
   \foreach \x in {1,2,4,5}
   {
     \node [v] (v\x) at (60*\x:1) {};
   }
   \foreach \x in {3}
   {
     \node [v,label=left:$v$] (v\x) at (60*\x:1) {};
   }
   \foreach \x in {6}
   {
     \node [v,label=right:$w$] (v\x) at (60*\x:1) {};
   }
   \draw (v1)--(v2)--(v3)--(v4)--(v5)--(v6)--(v1);
   \draw (v6)--(v3);
   \draw (v5)--(v1)--(v4)--(v2);
  \end{tikzpicture}
  \caption{An example showing that $e_{\mathcal O}$ may increase by contracting the edge $vw$.}
  \label{fig:ex2}
\end{figure}

This paper is organized as follows. 
Section~\ref{sec:prelim} reviews necessary definitions.
In Section~\ref{sec:general} we use the Erd\H{o}s-P\'osa property of connected subgraphs of graphs of bounded tree-width to prove several general useful lemmas. Theorems~\ref{thm:forest_edge} and~\ref{thm:diamond_edge} are derived from these lemmas in Section~\ref{sec:edge}. We investigate vertex-brittleness in Section~\ref{sec:vertex} and prove a characterization of minimal ideals with unbounded $\C$-vertex-brittleness for $\C \in \{\mc{A},\mc{D},\mc{O}\}$ in terms of \emph{traps}, a technical notion introduced in that section. In Section~\ref{sec:traps} we finish the proofs of Theorems~\ref{thm:forest_vertex},~\ref{thm:diamond_vertex}, and~\ref{thm:outer_vertex} by classifying the traps for the classes of acyclic, diamond-free,  and outerplanar graphs. The examples of graphs with properties claimed in Proposition~\ref{prop:example} are provided in Section~\ref{sec:example}. Section~\ref{sec:remarks} contains the concluding remarks.

\section{Preliminaries}\label{sec:prelim}
For graphs $G$ and $H$, $G\cup H$ and $G\cap H$ are graphs with 
	$(V(G\cup H), E(G\cup H))=(V(G)\cup V(H),E(G)\cup E(H))$  and $(V(G \cap H), E(G\cap H))=(V(G)\cap V(H), E(G)\cap E(H))$.
Here, $G\cup H$ and $G\cap H$ are called the \emph{union} and the \emph{intersection} of $G$ and $H$, respectively.
If $V(G)\cap V(H)=\emptyset$, then $G\cup H$ is called the \emph{disjoint union}  of $G$ and $H$.

For a set $F$ of edges of $G=(V,E)$, we write $G\setminus F$ for the subgraph $(V,E\setminus F)$, which is the subgraph obtained by deleting edges in $F$.
For a set $X$ of vertices of $G=(V,E)$, we write $G\setminus X$ for the subgraph obtained by deleting all vertices in $X$ and all edges incident with vertices in $X$.

For vertices $u$ and $v$ of a graph $G$, a \emph{path} from $u$ to $v$  is an alternating sequence $v_0e_1v_1\cdots e_k v_k$ of distinct vertices and edges of $G$
such that $k\ge 0$, $v_0=u$, $v_k=v$, and $e_i=v_{i-1}v_i$ for all $i\in\{1,2,\ldots,k\}$.
The \emph{length} of a path is its number of edges.
We sometimes identify such a path with the subgraph whose vertex set is $\{v_0,v_1,\ldots,v_k\}$ and edge set is $\{e_1,e_2,\ldots,e_k\}$.
A set $\mathcal P$ of paths is \emph{internally disjoint}
if no internal vertex of a path in $\mathcal P$
is on a path of another path in $\mathcal P$.

A graph is \emph{connected} if for all vertices $u$ and $v$, there is a path from $u$ to~$v$.
A graph is \emph{$k$-connected} if $\abs{V(G)}>k$ and 
and $G\setminus X$ is connected for all $X\subseteq V(G)$ with $\abs{X}<k$.

A set $X$ of vertices is \emph{independent} if no two vertices are adjacent.

We will use the following version of Menger's theorem in Section~\ref{sec:general}.
\begin{THM}[Menger~\cite{M1927}]\label{thm:Menger}
  Let $S$, $T$ be disjoint sets of vertices in a graph~$G$.
  Then $G$ has $k$ pairwise edge-disjoint paths from $S$ to $T$
  if and only if 
  there exists no set $X$ of vertices such that
  $S\subseteq X$, $T\cap X=\emptyset$, and 
  $G$ has less than $k$ edges having one end in $X$ and the other end not in $X$.
\end{THM}

\section{General lemmas}\label{sec:general}

In this section we use the standard toolkit on graphs of bounded tree-width  to prove a result (Corollary~\ref{c:tworoots}) allowing us to ``pin down'' the subgraphs of graphs $G$ which do not belong to a given monotone class $\C$, or to find within~$G$ a nicely structured subgraph $H$ which is far from $\C$ according to all our measures. It is the key ingredient in the proofs of all our theorems.

The following theorem appeared in 
Gy\'{a}rf\'{a}s and Lehel~\cite{GL1969} implicitly 
and Cockayne, Hedetniemi, and Slater~\cite{CHS1979} explicitly.
\begin{THM}\label{thm:subtree}
	Let $T$ be a tree, and $\mathcal{T}$ be a family of subtrees of $T$.
	For every positive integer $k$, \begin{itemize}
		\item either there exist $k$ pairwise vertex-disjoint members of $\mc{T}$, or
		\item there is a subset $Z\subseteq V(T)$ with $|Z|<k$ such that every member of $\mc{T}$ intersects $Z$.
	\end{itemize}
\end{THM}

A \emph{tree-decomposition} of a graph $G$ is a pair $(T, \beta)$, where $T$ is a tree and  $\beta$ is a function that assigns a subset of
vertices of $G$ to each vertex $t$ of $T$, such that for every $uv\in E(G)$, there exists $t\in V(T)$ with $\{u,v\}\subseteq \beta(t)$, and
for every $v\in V(G)$, the set $\{t:v\in\beta(t)\}$ induces a non-empty connected subgraph of $T$.
The \emph{width} of a tree-decomposition  $(T, \beta)$ is equal to $\max_{v \in V(T)}(|\beta(v)|-1)$, and the \emph{tree-width} of a graph $G$  is equal to the minimum width of a tree-decomposition of $G$. The following theorem characterizes graphs of large tree-width. 

\begin{THM}[Robertson and Seymour~\cite{RS1986}]\label{thm:planar}
	For every planar graph $G$, there exists a constant $w=w(G)$ such that every graph not containing $G$ as a minor has tree-width at most $w$.
\end{THM}

A standard application of Theorem~\ref{thm:subtree} gives the following.

\begin{LEM}\label{l:EP2} There exists a function $f(w,k)$ satisfying the following.
  For positive integers $w$ and $k$,
  if $G$ is a graph with tree-width at most $w$, and $\mc{H}$ is a family of $2$-connected subgraphs of $G$, then either \begin{itemize}
		\item there exist $v \in V(G)$ and $H_1,H_2,\ldots,H_k \in \mc{H}$ such that $V(H_i) \cap V(H_j) \subseteq \{v\}$ for all $1 \leq i < j \leq k$, or
		\item there exists $X \subseteq V(G)$ such that $|X| \leq f(w,k)$ and $|V(H) \cap X| \geq 2$ for every $H \in \mc{H}$.
	\end{itemize}  
\end{LEM}	

\begin{proof}
	Let $(T,\beta)$ be a tree-decomposition of $G$ of width at most $w$. For each $H \in \mc{H}$, let $A(H) = \{v \in V(T) : \beta(v) \cap V(H) \neq \emptyset \}$ for each $H \in \mc{H}$. As elements of $\mc{H}$ are connected, by Theorem~\ref{thm:subtree} either \begin{itemize}
		\item[(i)]  there exist $H_1,H_2,\ldots,H_k \in \mc{H}$ such that $A(H_i) \cap A(H_j)=\emptyset$ for all $1 \leq i < j \leq k$, or
		\item[(ii)] there exists $Z \subseteq V(T)$ with $|Z|<k$ such that $A(H) \cap Z  \neq \emptyset$ for every $H \in \mc{H}$.
	\end{itemize}  
	As (i) implies the first outcome of the lemma, we assume that (ii) holds.
	
	Let $Y=\bigcup_{z \in Z}\beta(z)$. Then $|Y| \leq (k-1)(w+1)$, and  $V(H) \cap Y \neq \emptyset$ for every $H \in \mc{H}$.
	Fix $y \in Y$, and let $\mc{H}(y)=\{H \in \mc{H} : y \in V(H) \}$. Applying Theorem~\ref{thm:subtree} to the collection of sets $A'(H)=\{v \in V(T) : \beta(v) \cap (V(H) - \{y\}) \neq \emptyset \}$ for $H \in \mc{H}(y)$,  we conclude that either
	\begin{itemize}
		\item[(i$'$)] there exist $H_1,H_2,\ldots,H_k \in \mc{H}$ such that $V(H_i) \cap V(H_j) \subseteq \{y\}$ for all $1 \leq i < j \leq k$, or	
			\item[(ii$'$)] there exists $Z(y) \subseteq V(T)$ with $|Z(y)|<k$ such that $A'(H) \cap Z(y) \neq \emptyset$ for every $H \in \mc{H}$.
	\end{itemize} 
	If (i$'$) holds for some $y \in Y$, then we obtain the first outcome of the lemma. Otherwise, (ii$'$) holds for every $y \in Y$. Let $X(y)=\bigcup_{z \in Z(y)} \beta(z)$, and let $X=
	Y \cup (\bigcup_{y \in Y}X(y))$. By construction we have  $|V(H) \cap X| \geq 2$ for every $H \in \mc{H}$, and $|X| \leq |Y|k(w+1) \leq k(k-1)(w+1)^2$. Thus $f(w,k)=k(k-1)(w+1)^2$ satisfies the lemma.
\end{proof}

The main result of this section is obtained by combining Lemma~\ref{l:EP2} and Theorem~\ref{thm:planar}. 
 
\begin{COR}\label{c:tworoots}
	Let $\mc{F}$ be a finite collection of $2$-connected graphs, at least one of which is subcubic and planar. Let $\mc{C}$ be the ideal consisting of all $\mc{F}$-free graphs. Then for every $\ell$ there exists $N=N(\mc{F},\ell)$  such that for every graph $G$ at least one of the following holds.
	\begin{itemize}
		\item[(i)]  $\fan(F,S,\ell)$ is a topological minor of $G$ for some $F$ obtained from a graph in $\mc{F}$ by subdividing at most one edge and  $S \subseteq V(F)$ with $|S| \leq 1$, or
		\item[(ii)] there exists $X \subseteq V(G)$ with $|X| \leq N$ such that $|V(J) \cap X|  \geq 2$ for every subgraph  $J$ of $G$ such that $J \not \in \mc{C}$.
	\end{itemize}  
\end{COR}

\begin{proof}

		Let $F_0 \in \mc{F}$ be subcubic and planar. Suppose that the disjoint union $\ell F_0$ of $\ell$ copies of $F_0$ is a minor of $G$. Then, as $F_0$ is subcubic, equivalently $\fan(F_0,\emptyset,\ell)$ is a topological minor of $G$  and so (i) holds. Thus we may assume that $\ell F_0$ is not a minor of $G$. Thus by Theorem~\ref{thm:planar} there exists $w=w(\ell,F_0)$ such that $G$ has tree-width at most $w$. 
		
		The upper bound on tree-width allows us to apply Lemma~\ref{l:EP2}. Rather than doing so directly it is convenient for us to start by considering the implications of the first outcome of this lemma. Assume that (i) does not hold.
		
		Let $\mc{J}$ be a collection of subgraphs of $G$, so that each $J \in \mc{J}$ is isomorphic to a subdivision of a graph in $\mc{F}$, and there exists $v \in V(G)$ such that $V(J) \cap V(J') \subseteq \{v\}$ for every pair of distinct $J,J' \in \mc{J}$.
		
		Let $\mc{J}_0 = \{J \in \mc{J} : v \not \in V(J)\}$. Then the graphs in $\mc{J}_0$ are pairwise vertex-disjoint. As $\fan(F,\emptyset,\ell)$ is not a topological minor of $G$ for every $F \in \mc{F}$, by our assumption, there are fewer than $\ell$ elements of $\mc{J}_0$ isomorphic to a subdivision of $F$. Thus $|\mc{J}_0| \leq \abs{\mc{F}}(\ell-1)$. 
		
    Let $\mc{J}_1 = \mc{J} - \mc{J}_0$. 
		Let $\mc{F'}$ be a finite collection of  graphs such that every graph obtained from an element of $\mc{F}$ by subdividing at most one edge is isomorphic to an element of $\mc{F'}$.    
    Let $s=|\mc{F'}|$ and let $m=\max_{F \in \mc{F'}}|V(F)|$.
    Note that for every $J \in \mc{J}_1$ there exist $F \in \mc{F}'$ and  an embedding $\phi: F \emb J$ such that  $\phi(u)=v$ for some vertex $u$ of $F$.
		We say that the pair $(F,u)$ for which such an embedding exists is a \emph{signature} of $J$. 
    If some pair $(F,u)$ is a signature of $\ell$ distinct elements of $\mc{J}_1$, then $\fan(F,\{u\},\ell)$ is a topological minor of $G$, a contradiction. 
    As there are at most $sm$ possible signatures, we conclude that $|\mc{J}_1| \le sm(\ell-1)$. Thus $|\mc{J}| \le (\abs{\mc{F}}+sm)(\ell-1)$.
		
		We now apply Lemma~\ref{l:EP2} to $G$, the family $\mc{H}$ of all subgraphs of $G$ isomorphic to a subdivision of a graph in $\mc{F}$, and $k=(\abs{\mc{F}}+sm)(\ell-1)+1$. As shown above the first conclusion of the lemma can not hold, and so there exists $X \subseteq V(G)$ such that $|X| \leq f(k,w)$ and $|X \cap V(J)| \geq 2$ for every $J \in \mc{H}$.  It follows that (ii) holds with $N=f(k,w)$.
\end{proof}

Note that  if the first outcome of Corollary~\ref{c:tworoots} holds for a graph $G$ then $\nu_{\mc{C}}(G) \geq \ell$ and $\kappa_{\mc{C}}(G) \geq \ell$ by Proposition~\ref{prop:topteb} and the following lemma.

\begin{LEM}\label{l:lower} Let $\mc{H}$ be an ideal. Let $G$ be a graph such that $G \not \in \mc{H}$ and let $S \subsetneq V(G)$ be independent. Then $\nu_{\mc{H}}(\fan(G,S,\ell)) \geq \ell$. If additionally $G \setminus S$ and $G$ are connected then $\kappa_{\mc{H}}(\fan(G,S,\ell)) \geq \ell$.
\end{LEM}	

\begin{proof} Let $F=\fan(G,S,\ell)$, and let $G_1,G_2,\ldots,G_{\ell}$ be the subgraphs of $F$ such that $G_i$ is isomorphic to $G$ for every $1\leq i \leq \ell$ and $V(G_i)\cap V(G_j) = S$ for $i \neq j$.
Since $G_1,G_2,\ldots,G_\ell$ are edge-disjoint, we deduce that  $\nu_{\mc{H}}(F) \geq \ell$	by definition.

Suppose now that $G_i \setminus S$ is connected for every $i$. Let $\mc{E}=(E_1,\ldots,E_n)$ be a partition of $E(F)$ such that $F[E_j] \in \mc{H}$ for every $1 \leq j \leq n$, where $F[E]$ denotes the subgraph of $F$ induced by the edges in $E$. Let $X=X(\mc{E})$ be the set of all the vertices of $F$ incident with edges in at least two different parts of $\mc{E}$. 
As $G_i$ is isomorphic to $F[E(G_i)]$ we have $F[E(G_i)] \not \in \mc{H}$ for every $1 \leq i \leq \ell$. Thus there exist $e,e' \in E(G_i)$ belonging to different parts of $\mc{E}$. As $G_i \setminus S$ is connected and $S$ is independent, there exists a path $P$ in $G_i \setminus S$ joining an end of $e$ to an end of $e'$. Then $V(P) \cap X \neq \emptyset$ and so $(V(G_i) -S) \cap X \neq \emptyset$. It follows that $|X(\mc{E})| \geq \ell$ for every partition $\mc{E}$ as above, and so  $\kappa_{\mc{H}}(\fan(G,S,k)) \geq \ell$.  
\end{proof}	
\section{Bounded $\nu_{\mc{C}}, e_{\mc{C}}, \eta_{\mc{C}}$}\label{sec:edge}

In this section we derive from Corollary \ref{c:tworoots} the following theorem, which generalizes Theorems~\ref{thm:forest_edge} and~\ref{thm:diamond_edge}.

\begin{THM}\label{thm:edge} Let $\mc{F}$ be a finite collection of $2$-connected graphs. Let $\mc{C}$ be an ideal consisting of all $\mc{F}$-free graphs, and suppose that $D \not \in \mc{C}$. Then for an ideal $\mc{G}$ the following are equivalent.
	 \begin{enumerate}
	 	\item[(i)] $\nu_{\mc{C}}$ is bounded on $\mc{G}$.
	 	\item[(ii)] $e_{\mc{C}}$ is bounded on $\mc{G}$.
	 	\item[(iii)] $\eta_{\mc{C}}$ is bounded on $\mc{G}$.
	 	\item[(iv)] $\Fan(F,S) \not\subseteq \mathcal G$ for every $F$ obtained from a graph in $\mc{F}$ by subdividing at most one edge and  $S \subseteq V(F)$ with $|S| \leq 1$, and
	 	$\mc{K}_{2,*} \not\subseteq \mathcal G$.
	 \end{enumerate}
\end{THM}
	
\begin{proof}
	By Observation~\ref{obs:basic}, (iii) implies (ii), and (ii) implies (i). 
	
	By Lemma~\ref{l:lower}, $\nu_{\mc{C}}$ is unbounded on $\Fan(F,S)$ for every $F \not \in \mc{C}$ and every $S \subseteq F$ with $|S|\le 1$. Moreover, $K_{2,3} \not \in \mc{C}$ and $\Fan(K_{2,3},S) \subseteq \mc{K}_{2,*} $ where $S$ is the set of degree-$3$ vertices of $K_{2,3}$. It follows that (i) implies (iv).
	
	It remains to show that (iv) implies (iii). Let  $\ell$ be a positive integer chosen so that $\fan(F,S,\ell) \not \in \mc{G}$ for every pair $F,S$ as in (iv) 
	and $K_{2,\ell} \not \in \mc{G}$. By Corollary~\ref{c:tworoots}, there exists $N$ such that for every graph $G \in \mc{G}$ there exists $X \subseteq V(G)$ with $|X| \leq N$ such that $|V(J) \cap X|  \geq 2$ for every subgraph  $J$ of $G$ with $J \not \in \mc{C}$.
	
	Given $G \in \mc{G}$ and $X \subseteq V(G)$ as above we will bound $\eta_{\mc{C}}(G)$ by a function of $N$ and $\ell$.
	 
	Fix $x \in X$, let $X' = X - \{x\}$, and suppose that there exists a collection of pairwise edge-disjoint paths $\mc{P}$ in $G$ with $|\mc{P}| \geq 2\ell|X|$ such that each $P \in \mc{P}$ has one end in $x$, the other end in $X'$, and is internally disjoint from $X'$. By the pigeonhole principle, there exist $x' \in X'$ and $\mc{P}' \subseteq \mc{P}$ such that $|\mc{P}'| \geq 2\ell+1$ and $x'$ is an end of every $P \in \mc{P}'$. Let $\mc{P}'' \subseteq \mc{P}'$  be chosen maximal so that the paths in $\mc{P}''$ are pairwise internally vertex-disjoint. 
	Then $|\mc{P}''|\leq \ell$, as otherwise $G$ contains a subdivision of $K_{2,\ell}$ contrary to our assumptions. 
	As every path in $\mc{P}' - \mc{P}''$ shares an internal vertex with some path in $\mc{P}''$, there exist $P \in \mc{P}''$ and distinct $P',P'' \in \mc{P}' - \mc{P}''$  such that each of $V(P')$ and $V(P'')$ contains an internal vertex of $P$.
	 Let $Q$ be a subpath of $P'$ chosen minimal so that $Q$ has one end $x$ and the other end in $V(P)-\{x\}$. 
	Let $R$ be a subpath  of $P''$ that is minimal so that $R$ has one end $x$ and the other end in $V(P) \cup V(Q) - \{x\}$. It is easy to see that $(P \cup Q \cup R) \setminus x'$ contains a subdivision of $D$, with $x$ corresponding to one of the degree-$3$ vertices of the diamond, and an end of either $Q$ or $R$ corresponding to the second one. Thus there exists a subdivision $J$ of a graph $D \in \mc{C}$, so that $|V(J) \cap X| \leq 1$, a contradiction. 
	
	Therefore, 
	the collection $\mc{P}$ satisfying the above assumptions does not exist. 
	By Theorem~\ref{thm:Menger}, there exists a set $W(x) \subseteq V(G)$ such that $W(x) \cap X = \{x\}$ and there are fewer than $2\ell|X|$ edges with one end in $W(x)$ and the other in $V(G)-W(x)$.
	Let $E(x)$ denote this set of edges.
	
	We are now ready to define a partition of $V(G)$ which will certify that $\eta_{\mc{C}}(G)$ is bounded. 
	Let $X=\{x_1,\ldots,x_{|X|}\}$, and define $V_i=W(x_i)- \bigcup_{j<i}W(x_j)$ for $1 \leq i < |X|$ and $V_{|X|}=V(G)-\bigcup_{j < |X|}W(x_j)$. Then $(V_1,\ldots,V_{|X|})$ is indeed a partition of $V(G)$, and $V_i \cap X = \{x_i\}$ for every~$i$. By the choice of $X$ it follows that $G[V_i]$ is $\mc{F}$-free for every~$i$, and so $G[V_i] \in \mc{C}$. Moreover, every edge of $G$ with ends in different parts of our partition belongs to  $\bigcup_{i=1}^{|X|-1}E(x)$. 
	Thus there are at most $(|X|-1)(2\ell |X|-1)$ such edges. 
	It follows that  $\eta_{\mc{C}}(G) \leq (N-1)(2\ell N-1)$ for every $G \in \mc{G}$ and (iii) holds. 	
\end{proof}	
\forestedge*
\begin{proof}%
	Theorem~\ref{thm:edge} applied with $\mc{F}=\{K_3\}$ implies that the conditions 1, 2, and 3 of Theorem~\ref{thm:forest_edge} are all equivalent to
	\begin{enumerate}
		\item[4$'$.] $\Fan(F,S) \not\subseteq \mathcal G$ for every $F$ obtained from  $K_3$ by subdividing at most one edge and  $S \subseteq V(F)$ with $|S| \leq 1$, and
		$\mc{K}_{2,*} \not\subseteq \mathcal G$.
	\end{enumerate}
	It remains to observe that the above condition is equivalent to condition 4 of Theorem~\ref{thm:forest_edge}. 
	Indeed, $\mc{K}_{2,*}=\Fan(K_3,S)$ for $S \subseteq V(K_3)$ with $|S|=2$. Moreover, every ideal $\Fan(F,S)$ described in the condition above contains an ideal   $\Fan(K_3,S')$ for some  $S' \subseteq V(K_3)$ with $|S'| \leq 1$.
\end{proof}	
\diamondedge*
\begin{proof}%
	Our argument is essentially identical to the proof of Theorem~\ref{thm:forest_edge} above. 	Theorem~\ref{thm:edge} applied with $\mc{F}=\{D\}$ implies that the conditions 1, 2, and 3 of Theorem~\ref{thm:diamond_edge} are all equivalent to
	\begin{enumerate}
		\item[4$'$.] $\Fan(F,S) \not\subseteq \mathcal G$ for every $F$ obtained from  $D$ by subdividing at most one edge and  $S \subseteq V(F)$ with $|S| \leq 1$, and
		$\mc{K}_{2,*} \not\subseteq \mathcal G$.
	\end{enumerate}
	It is trivial that the condition 4$'$ implies the condition 4 of Theorem~\ref{thm:diamond_edge}.
	Observe that for every pair $(F,S)$ as in the above condition we have  $\Fan(D,S') \subseteq \Fan(F,S)$ for some  $S' \subseteq V(D)$ with $|S'| \leq 1$. 
	It follows that the above condition 4$'$ is equivalent to the condition 4 of Theorem~\ref{thm:diamond_edge}, as desired.
\end{proof}	

\section{Bounded vertex-brittleness}\label{sec:vertex}

In this section we prove a technical characterization of minimal ideals with unbounded $\C$-vertex-brittleness.

\begin{LEM}\label{lem:manypath}
	Let $n$ be an integer.
	Let $G$ be a graph and $u,v$ be non-adjacent vertices of $G$ such that $G\setminus \{u,v\}$ is connected.
	If $G$ has $3n$ internally disjoint paths from  $u$ to $v$, 
	and every subgraph of $G$ isomorphic to a subdivision of $K_{2,3}$ contains both $u$ and $v$,
	then  $G$ has a subgraph isomorphic to a subdivision of $\fan(K_{2,3},S,n)$, where $S$ consists of two degree-$2$ vertices of $K_{2,3}$.
\end{LEM}

\begin{proof}
	Let $\Gamma_1,\Gamma_2,\ldots,\Gamma_{3n}$ be internally disjoint paths in $G$ from $u$ to $v$, and let $\Gamma_i'=\Gamma_i \setminus \{u,v\}$ for $i = 1,2,\ldots,3n$.
	Since $u$ and $v$ are non-adjacent, $\Gamma_i'$ contains at least one vertex.
	
	Let $G'$ be the graph obtained from $G\setminus \{u,v\}$ by contracting $\Gamma_i'$ for all $i=1,2,\ldots,3n$. 
	Let $z_i$ be the new vertex in $G'$ obtained by contracting $\Gamma_i'$, and let $Z=\{z_1,z_2,\ldots,z_{3n}\}$.
	Since $G\setminus \{u,v\}$ is connected, so is $G'$.
	Let $T$ be a minimal tree in $G'$ containing all vertices in $Z$.
	Since $T$ is minimal, every leaf of $T$ is contained in  $Z$.

	We first claim that  $T$ is a path.
	
	Suppose $T$ contains a vertex $x$ of degree at least $3$.
	We choose three vertices $z_{i_1}, z_{i_2}, z_{i_3} \in Z$ such that the paths in $T$ joining $x$ and $z_{i_1}, z_{i_2},z_{i_3}$ are internally disjoint, and have no internal vertices in $Z$.
	We may assume $z_{i_1}=z_1$, $z_{i_2}=z_2$ and $z_{i_3}=z_3$ by relabelling $\Gamma_1,\Gamma_2,\ldots,\Gamma_{3n}$ if necessary.
	
	For $j=1,2,3$, let $Q_j$ be the path of $G$ consisting of the edges of the path from $x$ to $z_j$ in $T$.
	Assume first that $x\in Z$, then, without loss of generality, let $x=z_4$.
	It is easy to see that 
	$Q_1$, $Q_2$, and $Q_3$ are 
	internally disjoint. %
	Let $x_j$ be the end vertex of $Q_j$ in $V(\Gamma_4')$. %

	We may assume by permuting $\Gamma_1$, $\Gamma_2$, and $\Gamma_3$ if necessary that $x_2$ is contained in the subpath of $\Gamma_4'$ from $x_1$ to $x_3$. 
	Note that $x_1$, $x_2$, and $x_3$ are not necessarily distinct.
	For $j=1,2,3$, 
	the union
	$\Gamma_j \cup Q_j \cup \Gamma_4'$ is a tree, and so it contains a unique path from $u$ to $x_2$.
	The three paths form a subdivision of $K_{2,3}$ where $x_2$ and $u$ are the degree-$3$ vertices in $G \setminus v$, a contradiction to the assumption that every subgraph of $G$ isomorphic to a subdivision of $K_{2,3}$
	contains both $u$ and $v$.
	
	So, $x\notin Z$.
	In this case, $\Gamma_j \cup Q_j$ contains a path from $u$ to $x$ disjoint from $v$ for every $j=1,2,3$. The resulting paths are internally disjoint and so $G \setminus v$ once again contains   a subdivision of $K_{2,3}$, a contradiction, finishing the proof of the claim.

	By relabelling if necessary, we may assume that $z_1,z_2,\ldots,z_{3n}$ lie on $T$ in this order.
	For $i=1,2,\ldots,3n-1$, 
	let $W_i$ be the path in $G\setminus \{u,v\}$ corresponding to the path in $T$ from $z_i$ to $z_{i+1}$,
	and let $x_i$ and $y_i$ be the ends of $W_i$ in $\Gamma_i'$ and $\Gamma_{i+1}'$, respectively.
	Clearly, $W_1,W_2,\ldots,W_{3n-1}$ are internally disjoint.
	For each $j=1,2,\ldots,n$, 
	the union of $\Gamma_{3j-2}, \Gamma_{3j}$, $W_{3j-2}$, $W_{3j-1}$ and the subpath in $\Gamma_{3j-1}$ from $y_{3j-2}$ to $x_{3j-1}$ forms a subdivision $R_j$ of $K_{2,3}$.
	Then, the union of $R_1,R_2,\ldots,R_n$ is isomorphic to 
	the desired subdivision of $\fan(K_{2,3},S,n)$, where $S$ is as in the lemma statement.
\end{proof}

	 For a graph $G$  and  a degree-$2$ vertex $v$ of $G$, we denote by $G/v$ the graph obtained from $G$ by contracting one of the edges incident with $v$. Note that $G/v$ is a topological minor of $G$.
 
	We now present the main technical definition of this section.
	Let $H$ be a graph. 
	An \emph{$H$-snare} is a pair $(J,S)$ of a graph $J$ and $S\subseteq V(J)$ such that $H$ is a topological minor of $J$, $S$ is an independent set in $J$, and $J \setminus S$ is connected.
  We say that $(J,S)$ is an \emph{$H$-trap} if $(J,S)$ is an $H$-snare, and $(J,S)$ is minimal in the following sense: $(J', S \cap V(J'))$ is not an $H$-snare for every proper subgraph $J'$ of $J$, and $(J/v,S)$ is not an $H$-snare for every vertex $v \in V(J)-S$ of degree two. 
  We will characterize minimal ideals with unbounded $\C$-vertex-brittleness in terms of $\Fan(J,S)$ for traps $(J,S)$ in Theorem~\ref{t:vbrittle}.

The next lemma shows that for every graph $H$ the size of every $H$-trap $(J,S)$ is bounded by a function of $H$
and $|S|$.

\begin{LEM}\label{l:boundtrap} Let $H$ be a connected graph and let $(J,S)$ be an $H$-trap. Then $|V(J)| \leq 5|E(H)|+4|V(H)|+9|S|$.
\end{LEM}

\begin{proof}
	Let $\phi: H \emb J$  be an embedding, and let $H'=\phi(H)$. Let $C_1$, $C_2$, $\ldots$, $C_m$ be all the components of $H' \setminus S$. 
	If $m = 1$ then by the definition of an $H$-trap, $J=H'$. In this case it is easy to see that $|\phi(e)\cap S|\le 1$ and  $|V(\phi(e))| \leq 3$ for every $e \in E(H)$. So the lemma holds. 
	Thus we assume that $m \geq 2$.
	
 	Note that for every component $C$ of $H' \setminus S$ either we have $\phi(v) \in V(C)$ for some $v \in V(H)$, or  $C$ is a subgraph of $\phi(uv) \setminus \{\phi(u), \phi(v)\}$ for some $uv \in E(H)$. In the second case, either  $C = \phi(uv) \setminus \{\phi(u), \phi(v)\}$, or there exists $s \in S \cap (V(\phi(uv)) - \{\phi(u), \phi(v)\})$ such that $s$ has a neighbor in $C$. As every such vertex $s$ has degree two in $H'$, we conclude that
        \[m \leq |V(H)|+|E(H)|+2|S|.\]
	
	Let $J'$ be a subgraph of $J \setminus S$ chosen so that $J'$ is connected, $V(J') \cap V(C_i) \neq \emptyset$ for every $1 \leq i \leq m$,  $|E(J') - E(H')|$ is minimum, and subject to that $|E(J')|$ is minimum.
	
	Then $(H' \cup J', S \cap V(H'))$ is an $H$-snare, and so $J=H' \cup J'$. Moreover, $J'$ is a tree by minimality of $E(J')$.
	
	Suppose for a contradiction that $J' \cap C_i$ is disconnected for some $1 \leq i \leq m$. Then there exists a cycle $F$ in $J' \cup C_i$ such that $E(F) - E(H') \neq \emptyset$. 
	Then $J'' = (J' \cup C_i) \setminus e$ contradicts the choice of $J'$ for every $e \in E(F)-E(H')$.  Thus $J' \cap C_i$ is  connected for every $1 \leq i \leq m$.
	
	Let $J_i = J' \cap C_i$, and let $T$ be a tree obtained from $J'$ by contracting $J_i$ to a single vertex for each $i$. 
	Then every leaf or degree-$2$ vertex of $T$ corresponds to $J_i$ for some $i$ by the choice of $J'$ and the definition of an $H$-trap. 
	For every tree, the number of vertices of degree at least three in the tree is less than that of leaves.
	Since $T$ has at most $m$ vertices of degree at most $2$, we deduce that $|V(T)| \leq 2m$, implying that $|E(J')-E(H')| \leq 2m$.
	
	We claim that
        if a vertex $v$ of $J$ is not in $\phi(V(H))\cup S$, then
        it is incident with an edge in $E(J')-E(H')$.
        Suppose not.
	If $\deg_J(v) \ge 3$, then since $v$ is not incident with any edge in $E(J')-E(H')$ and $J=H' \cup J'$, 
	 $\deg_{H'}(v)\ge 3$, which implies that $v \in \phi(V(H))$, a contradiction.
	So, $\deg_J(v) \le 2$. If $\deg_J(v)\le 1$, then since $v\notin \phi(V(H))$, $(J\setminus v,S-\{v\})$ is an $H$-snare, a contradiction.
	If $\deg_J(v)=2$, then $v \in \phi(uw)\setminus \{\phi(u), \phi(w)\}$ for some $uw\in E(H)$. 
	And since $v\notin S$, $(J/v, S)$ is an $H$-snare, a contradiction.
	Therefore, the claim holds.

	By the above claim,
	we have
	\begin{multline*}
	|V(J)|\le 2|E(J')-E(H')|+|V(H)|+|S|\\
		 \le  4m+|V(H)|+ |S|\le 5|V(H)|+4|E(H)|+9|S|.
	\end{multline*}
	This completes the proof.
\end{proof}
		
The next theorem is the main result of this section and the key step in the proof of Theorems~\ref{thm:forest_vertex},~\ref{thm:diamond_vertex} and~\ref{thm:outer_vertex}.				

\begin{THM}\label{t:vbrittle} Let $\mc{F}$ be a finite collection of $2$-connected graphs. Let $\mc{C}$ be an ideal consisting of all $\mc{F}$-free graphs, and suppose that $K_{2,3} \not \in \mc{C}$. Then $\kappa_{\mc{C}}$ is bounded on an ideal $\mc{G}$ if and only if for every $F \in \mc{F}$ and every $F$-trap $(J,S)$ we have $\Fan(J,S) \not\subseteq \mc{G}$. 
\end{THM}

\begin{proof}
For every $F \in \mathcal{F}$ and every $F$-trap $(J,S)$, $J$ is connected since $F$ is connected. 
So, by Lemma~\ref{l:lower}, 
$\kappa_{\mc{C}}$ is unbounded on $\Fan(J,S)$  for every $F \in \mc{F}$ and every $F$-trap   $(J,S)$. 
This implies the ``only if'' part of the theorem statement. 

Now we prove the ``if'' part.
Suppose  $\Fan(J,S) \not\subseteq \mc{G}$  for every $F \in \mc{F}$ and every $F$-trap $(J,S)$.

Let $F'$ be a graph obtained from a graph in $\mc{F}$ by subdividing at most one edge and  $S' \subseteq V(F')$ with $|S'| \leq 1$.
Then,  
either 
$(F',S')$ is an $F$-trap or
there exists $S \subseteq V(F)$ with $|S|\le 1$ such that $FAN(F,S) \subseteq FAN(F',S')$,
which implies that
$FAN(F',S') \not\subseteq \mathcal{G}$ since for every $F\in \mathcal{F}$ and $S\subseteq V(F)$ with $|S|\le 1$, $(F,S)$ is an $F$-trap. %
Thus,  there exists a positive integer $\ell$ 
such that $\fan(F',S',\ell) \not \in \mc{G}$ for every $F'$ obtained from a graph in $\mc{F}$ by subdividing at most one edge and  every $S' \subseteq V(F')$ with $|S'| \leq 1$

As 	 $\mc{F}$ contains a topological minor of $K_{2,3}$, Corollary~\ref{c:tworoots} is applicable and thus there exists $N=N(\mc{F},\ell)$ such that every $G \in \mc{G}$  has a subset $X \subseteq V(G)$ such that $|V(J) \cap X|  \geq 2$ for every subgraph  $J$ of $G$ with $J \not \in \mc{C}$.

By Lemma~\ref{l:boundtrap} there exists an absolute bound on the number of vertices of $J$ for each $F\in \mc{F}$ and each $F$-trap $(J,S)$ with $|S|\le N$.
Thus there exist $R = R(\mc{F},N)$ and a collection $\mc{J}_* = \{(J_1,S_1),\ldots,(J_R,S_R)\}$ such that every $F$-trap $(J,S)$ with $|S|\le N$ is isomorphic to an element of $\mc{J}_*$. That is, explicitly, $\mc{J}_*$ satisfies the following: for every 
$F \in \mc{F}$ and every $F$-trap $(J,S)$ with $|S|\le N$,  there exist $(J_i,S_i) \in \mc{J}_*$ and an isomorphism $\psi: V(J_i) \to V(J)$ between $J_i$ and $J$ such that $\psi$ maps $S_i$ bijectively on to $S$. By our assumption there exists a positive integer $\ell'$  so that $\fan(J_i,S_i,\ell') \not \in \mc{G}$ for every $(J_i,S_i) \in \mc{J}_*$.

Consider now $G \in \mc{G}$ and $X$ as obtained from   Corollary~\ref{c:tworoots} above. An \emph{$X$-bridge} in $G$ is a maximal subgraph $B$ of $G$ such that $B$ and $B \setminus X$  are connected, and $X \cap V(B)$ is independent in $B$. Thus every $X$-bridge consists of a component $C$ of $G \setminus X$ together with all the neighbors of vertices of~$C$ in $X$ and edges from $C$ to these neighbors. Conversely every component of $G \setminus X$ gives rise to a unique $X$-bridge. Let $\mc{B}$ be the collection of all $X$-bridges $B$ in $G$ such that $B \not \in \mc{C}$.

We bound $|\mc{B}|$ as follows. 
For every $B \in \mc{B}$  there exists $F \in \mc{F}$ so that $(B, X \cap V(B))$ is an $F$-snare. 
Thus there exist an $F$-trap $(J(B),S(B))$ and an embedding $\phi_B: J(B) \emb B$ 
such that 
 $V(\phi_B(J(B))) \cap X=\phi(S(B))$. 
Then, there exist $(J_i,S_i)\in \mc{J}_*$ and an isomorphism  $\psi_B: V(J_i) \to V(J(B)) $ mapping $S_i$ onto $S(B)$. We define the signature of $B$ to be a pair of $i$ and the restriction of the map $\phi \circ \psi_B$ to  $S_i$. %
Thus there are at most $R|X|! \leq R \cdot N!$ possible signatures. Moreover, note that if $X$-bridges $B_1,B_2,\ldots, B_{\ell'}$ all have the same signature $(i,\pi_i)$, then combining $\phi_{B_1},\ldots,\phi_{B_{\ell'}}$ one can define a natural embedding of $\fan(G_i,S_i,\ell')$ into $G$, contradicting the choice of $\ell'$. The pigeonhole principle implies that $|\mc{B}| \leq \ell'R N!$.

Next we use Lemma~\ref{lem:manypath} to break up every $B \in \mc{B}$. 
Consider distinct $u,v \in V(B) \cap X$ and let $B' = B \setminus (X - \{u,v\})$. By the choice of $X$ every subdivision of $K_{2,3}$ in $B'$ contains both $u$ and $v$. Moreover, $\fan(K_{2,3},S,\ell')$ is not a subgraph of $B'$. Thus by Lemma~\ref{lem:manypath} $B'$ does not contain $3\ell'$ internally disjoint paths from $u$ to $v$. 
By Menger's theorem,  there exists a set $Y(B,u,v) \subseteq V(B)-X$ with $|Y(B,u,v)| \leq 3\ell'$ so that $u$ and $v$ belong to different components of $B' \setminus Y(B,u,v)$.

Let $Y$ be the union of $X$ and all sets $Y(B,u,v)$ as above. 
Then $|Y| \leq N+3\ell'\binom{N}{2}|\mc{B}| \le N+2N^2|\mc{B}|\ell'$. 
Let $\mc{E}$ be the partition of $E(G)$ into edge sets of $Y$-bridges in $G$ and the one-element parts corresponding to edges with both ends in $Y$. Then every vertex incident with edges in two distinct parts of $\mc{E}$ belongs to $Y$ by construction. Moreover, every $Y$-bridge is either an $X$-bridge $B$ such that $B \in \mc{C}$, or a part of an $X$-bridge in $\mc{B}$ in which case it contains no path between two distinct elements of $X$, by construction of $Y$. Thus every $Y$-bridge belongs to $\mc{C}$. It follows that $\kappa_{\mc{C}}(G) \leq N+2N^2|\mc{B}|\ell' \leq N+2N^2 R(N!)(\ell')^2$, and so   $\kappa_{\mc{C}}$ is bounded on $\mc{G}$.  
\end{proof}

\section{Classifying traps}\label{sec:traps}

Theorem~\ref{t:vbrittle} reduces the proofs of Theorems~\ref{thm:forest_vertex},~\ref{thm:diamond_vertex} and~\ref{thm:outer_vertex} to the problem of classification of respective traps, which is the goal of this section.

As $\mc{A}$ consists of all $K_3$-free graphs the next lemma implies Theorem~\ref{thm:forest_vertex}.  

\begin{LEM} If $(J,S)$ is a $K_3$-trap then $J$ is isomorphic to $K_3$ and $|S| \leq 1$.
\end{LEM}

\begin{proof}
As $J$ contains $K_3$ as a topological minor, there exists a cycle $C$ in $J$. If $V(C) \cap S = \emptyset$ then $(C,\emptyset)$ is a $K_3$-snare and so by minimality $J=C$ and $\abs{V(C)}=3$ as desired.
Otherwise, there exists $s \in V(C) \cap S$. Let $u,v$ be the neighbors of $s$ in $C$. As $S$ is independent we have $u,v \in V(J) - S$. As $J \setminus S$ is connected there exists a path $P$ from $u$ to $v$ in $J \setminus S$. Adding $s$ and edges $su$ and $sv$ to $P$ we obtain a cycle $C'$ in $J$ such that $V(C') \cap S = \{s\}$. As $(C',\{s\})$ is a $K_3$-snare, we once again have $J=C'$ and $|V(C')|=3$.	
\end{proof}

Similarly, Theorem~\ref{thm:diamond_vertex} is implied by the following.
  
\begin{LEM} If $(J,S)$ is a $D$-trap then $J$ is isomorphic to $D$.
\end{LEM}

\begin{proof}
Let $H$ be a subgraph of $J$ isomorphic to a subdivision of $D$, chosen so that $|V(H) \cap S|$ is minimum.	Let $S'=V(H) \cap S$. Let $u,v$ be the two vertices of $H$ of degree three and let $P_1,P_2,P_3$ be the internally disjoint paths from $u$ to $v$ such that $H = P_1 \cup P_2 \cup P_3$ 

Suppose first that $H \setminus S'$ is connected. Then $(H,S')$ is a $D$-snare and so $H=J$ by minimality. 
If $|S'| \leq 1$ then $J$ is  isomorphic to $D$,  as otherwise one can choose a degree-$2$ vertex in $V(J) - S'$ to suppress so that the resulting graph is still a subdivision of $D$, contradicting minimality of $(J,S)$. 
If $|S'|=2$ then no $P_i$ contains $S'$ since otherwise $J\setminus S'$ is not connected.
So, similar to the case that $|S'|\le 1$, one can prove that $J$ is isomorphic to $D$. 

Suppose now for a contradiction that $H \setminus S'$ is not connected. Then either
\begin{itemize}
	\item $|V(P_i) \cap S'| \geq 2$ for some $1 \leq i \leq 3$, or
	\item some internal vertex of $P_i$ lies in $S'$ for every $1 \leq i \leq 3$.
\end{itemize}

We start considering the first case. Suppose without loss of generality that there exist distinct $s_1,s_2 \in V(P_1) \cap S'$. Let $Q$ denote the subpath of $P_1$ with ends $s_1$ and $s_2$. By connectivity of $J \setminus S$ there exists a path $R$ in $J \setminus S$ with one end in $x \in V(Q)-S'$, the other end $y \in V(H) - V(Q)-S'$, internally disjoint from $H$. 

If $y \in V(P_1)$ then replacing the subpath of $P_1$ with ends $x$ and $y$ by~$R$, we obtain a subdivision  $H'$ of $D$ such that $|V(H') \cap S| < |V(H) \cap S|$, contradicting the choice of $H$. If $y \not \in V(P_1)$, then we obtain a similar contradiction, this time by  deleting the internal vertices of the subpath of~$P_1$ with ends $x$ and $v$, and adding $R$. This finishes the analysis of the first case.

For the second case let $s_i \in S'$ be an internal vertex of $P_i$ for $i=1,2,3$. Let $C$ and $C'$ be the components of $H \setminus \{s_1,s_2,s_3\}$ containing $u$ and $v$, respectively. Since $J\setminus S$ is connected, there exists a path $R$ in $J \setminus S$ with one end $x \in V(C)-S$, the other end $y \in V(C')-S$, internally disjoint from $H$. As in the previous case we now obtain a contradiction by rerouting $H$ along $R$. If $x,y \in V(P_i)$ for some $1 \leq i \leq 3$, then we replace the subpath of $P_i$ with ends $x$ and $y$ by $R$. Otherwise, assuming $x \in V(P_i)$ we, as before, delete the internal vertices of the subpath of $P_i$ with ends $x$ and~$v$, and add $R$. In both cases we obtain a subdivision $H'$ of $D$ such that $|V(H') \cap S| < |V(H) \cap S|$, a contradiction.  
\end{proof}

Finally, the next lemma contains the technical part of the proof of Theorem~\ref{thm:outer_vertex}: A classification of $K_{2,3}$-traps.
For a set $X$ of vertices, an \emph{$X$-path} in $G$ is a path in $G$ with both ends in $X$  and no internal vertices in $X$.

\begin{LEM}\label{l:outer_trap} 
If $(J,S)$ is a $K_{2,3}$-trap, then either
\begin{itemize}
\item $J$ is isomorphic to $K_{2,3}$ or 
\item $J$ is isomorphic to either $K_{2,3}^{+}$ or $W^+_k$ for some $k \geq 3$ and $S$ is the set of all degree-$2$ vertices of $J$.
\end{itemize}
\end{LEM}

\begin{proof}
Assume for a contradiction that $(J,S)$ is a $K_{2,3}$-trap for which the conclusion of the lemma does not hold.
We choose a subgraph $H$ of $J$  isomorphic to a subdivision of $K_{2,3}$ 
that minimizes the number of vertices of $H$ in $S$, and subject to that the number of degree-$3$ vertices of $H$ in $S$ is minimized. 
Let $a$, $b$ be the degree-$3$ vertices in $H$, 
and let $\Gamma_1,\Gamma_2,\Gamma_3$ be internally disjoint paths in $H$ joining $a$ and $b$.
To simplify our presentation, 
we will write a \emph{subpath} in $H$ to denote a subpath of $\Gamma_1$, $\Gamma_2$, or $\Gamma_3$.
For distinct vertices $s,t$ on the same path $\Gamma_i$ for some $i=1,2,3$ where $\{s,t\}\neq \{a,b\}$, 
if $s$ is not adjacent to $t$, then let $I(s,t)$ be the set of vertices in the component of $H\setminus \{s,t\}$ not containing $a$ or $b$,
and otherwise,  we let $I(s,t)=\emptyset$.
Let $O(s,t)=V(H)- \left(\{s,t\} \cup I(s,t)\right)$.
We remark that 
if $s,t \in S$, 
then $I(s,t)-S\neq \emptyset$ because $S$ is an independent set.
\medskip

\noindent (1) $H \setminus S$ is not connected.
\medskip

If $H \setminus S$ is connected, then $(H, V(H) \cap S)$ is a $K_{2,3}$-snare. It follows from minimality of $(J,S)$ that $J=H$.
If for some $i$, $\Gamma_i$ contains at least two  degree-$2$ vertices, then it has a degree-$2$ vertex $v$ not in $S$,
which implies that $(H/v,S)$ is a $K_{2,3}$-snare, a contradiction to the minimality of $(J,S)$. 
Since $J=H$ is a subdivision of $K_{2,3}$, $\Gamma_i$ has length $2$ for $i=1,2,3$, which implies that $J$ is isomorphic to $K_{2,3}$, a contradiction to our assumption.
\bigskip

\noindent (2) If $I(x,y) \cap S\neq \emptyset$ for some distinct $x,y \notin S$ on the same path $\Gamma_i$ for some $i=1,2,3$, and $\{x,y\}\neq \{a,b\}$, 
then $J\setminus S$ has no $V(H)$-path joining $x$ and~$y$.
\medskip

Suppose there exists such a path $Q$.
We consider the subgraph $H'$ of $J$ obtained from $H$ by removing $I(x,y)$ and adding $Q$.
Clearly, $H'$ is isomorphic to a subdivision of $K_{2,3}$.
However, $|V(H')\cap S| <|V(H)\cap S|$ 
since $Q$ contains no vertex in $S$, but $I(x,y)$ contains a vertex in $S$. 
This yields a contradiction to the minimality of the number of vertices of $H$ in $S$.
This proves (2).
\bigskip

\noindent (3) If $J\setminus S$ has a $V(H)$-path $Q$ from $x$ to $y$,
and $x$ and $y$ are not on the same path $\Gamma_i$ for $i=1,2,3$, 
then $I(a,x)\cap S =\emptyset$ or $y$ is adjacent to $b$.
\medskip

Suppose $I(a,x) \cap S \neq \emptyset$ and $I(y,b)$ has at least one vertex.
Then, the subgraph $H'$ of $J$ obtained from $H$ by removing $I(a,x)$ and adding $Q$ is isomorphic to a subdivision of $K_{2,3}$
and 
$|V(H')\cap S| < |V(H)\cap S|$
since every internal vertex of $Q$ is not in $S$ but $I(a,x)$ contains a vertex in~$S$. 
This is a contradiction to the minimality of the number of vertices of $H$ in~$S$. This proves (3).
\bigskip

\noindent (4) $a,b \notin S$.
\medskip

Suppose not.
Without loss of generality, we may assume $a\in S$.
Suppose $V(H)\cap S \not\subseteq \{a,b\}$.
Then,
there is a vertex  $s\in (V(H)\cap S)- \{a,b\}$
such that $I(a,s)$ has no vertex in $S$.

Since $J\setminus S$ is connected, there is a $V(H)$-path $Q$ in $J\setminus S$
connecting $I(a,s)$ and $O(a,s)$.
Let $x$ and  $y$ be the end vertices of $Q$ in $I(a,s)$ and $O(a,s)$, respectively.
By (2), $y$ is not on the same path $\Gamma_i$ with $s$.
Let $H'$ be the subgraph of $J$ obtained from $H$ by removing $I(a,y)$ and adding $Q$.
Then $H'$ is isomorphic to a subdivision of $K_{2,3}$.
In addition, $|S\cap V(H')|\le |S\cap V(H)|$  since $Q$ contains no vertices in $S$, 
and 
the number of degree-$3$ vertices of~$H'$ in $S$ is less than that of $H$ since $x$ and $b$ are the degree-$3$ vertices of~$H'$ and $x \notin S$,
contradicting the assumption on the choice of $H$.
Hence, $S\cap V(H) \subseteq \{a,b\}$, and thus $S \cap V(H)=\{a,b\}$  by (1).

\begin{figure}
	\centering
	\begin{tikzpicture}
	\tikzstyle{v}=[circle, draw, solid, fill=black, inner sep=0pt, minimum width=3pt]
	\draw (0,0) node [v,label=$a$] (a){};
	\draw(3.5,0) node[v,label=$b$](b){};
	\draw (a) to  node[pos=0.2,label=below:$\Gamma_2$,inner sep=0pt] {}node[pos=0.5,v,label=below:$t_1$](t1){}  (b);
	\draw (a) [out=60,in=120] to node[inner sep=0pt,pos=0.2,label=$\Gamma_1$] {}
	node[pos=0.5,v,label=$s_1$](s1){} (b);
	\draw (a) [out=-60,in=-120] to node[inner sep=0pt, pos=0.2,label=below:$\Gamma_3$] {}
	node[pos=0.7,v,label=below:$s_2$](s2){} (b);
	\draw (s1)-- node[pos=0.5,v,label=left:$t_2$](t2){} (t1);
	\draw (t2) [out=0,in=90] to (s2);
	\end{tikzpicture}
	\quad
	\begin{tikzpicture}
	\tikzstyle{v}=[circle, draw, solid, fill=black, inner sep=0pt, minimum width=3pt]
	\draw (0,0) node [v,label=$a$] (a){};
	\draw(3.5,0) node[v,label=$b$](b){};
	\draw (a) to  node[pos=0.2,label=below:$\Gamma_2$,inner sep=0pt] {}
	node[pos=0.7,v,label=$t_2$](t2){} 
	node[pos=0.5,v,label=below:$t_1$](t1){}  (b);
	\draw (a) [out=60,in=120] to node[inner sep=0pt,pos=0.2,label=$\Gamma_1$] {}
	node[pos=0.5,v,label=$s_1$](s1){} (b);
	\draw (a) [out=-60,in=-120] to node[inner sep=0pt, pos=0.2,label=below:$\Gamma_3$] {}
	node[pos=0.7,v,label=below:$s_2$](s2){} (b);
	\draw (s1)-- (t1);
	\draw (t2) to (s2);
	\end{tikzpicture}
	
	\caption{Two possible arrangements of two $V(H)$-paths $Q_1$ and $Q_2$ in the proof of Lemma~\ref{l:outer_trap} .}
	\label{fig:q1q2}
\end{figure}
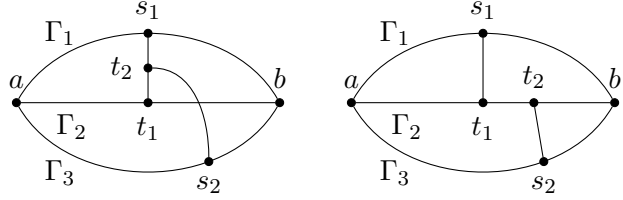

Let $Q_1$ be a $V(H)$-path in $J\setminus S$ joining two components of $H\setminus S$, and let $H_1$ be the union of $H$ and $Q_1$. 
We may assume that $Q_1$ has one end in $s_1 \in V(\Gamma_1)$ and the other end in $t_1 \in V(\Gamma_2)$.

Let $Q_2$ be a $V(H_1)$-path in $J\setminus S$ with one end in $s_2 \in V(\Gamma_3)$ joining the two components of $H_1 \setminus S$. Let $t_2$ be the second end of $Q_2$.

See Figure~\ref{fig:q1q2} for an illustration.
If $t_2$ is an internal vertex of $Q_1$, then $(H_1 \cup Q_2) \setminus \{b\}$ has three internally disjoint paths of length at least two from $a$ to $t_2$.
They form a subdivision of $K_{2,3}$ containing only one vertex in $S$, a contradiction.
Hence, $t_2$ is contained in $V(H)$, 
and we may assume that $t_2\in V(\Gamma_2)$, and 
we may further assume that $t_2$ is not closer to $a$ than $t_1$ in $\Gamma_2$.
Note that $t_1$ and $t_2$ are not necessarily distinct.
Then, there are three internally disjoint paths of length at least $2$ in the union of $H$, $Q_1$ and $Q_2$ from $s_1$ to $s_2$ not using any vertex in $I(a,t_1)$ or $I(t_2,b)$.
They form a subdivision $K$ of $K_{2,3}$ containing exactly two vertices in $S$  
not both of which have degree three in $K$, contradicting the choice of $H$. This proves (4).
\bigskip

\noindent(5) If a $V(H)$-path in $J\setminus S$ from $x$ to $y$ joins distinct components of $H\setminus S$ 
and $(I(a,x)\cup I(b,x))\cap S \neq \emptyset$, then $y$ is adjacent to $a$ and $b$, and both $I(a,x)$ and $I(b,x)$ contain vertices in $S$.
\medskip

Without loss of generality, suppose $I(a,x)\cap S \neq \emptyset$. 
By (3), we know that $y$ is adjacent to $b$. 
Since $x$ and $y$ are in distinct components of $H\setminus S$,
$S \cap I(x,b)\neq \emptyset$. 
Then, by (3) again, $y$ is adjacent to $a$. This proves (5).
\bigskip

\noindent(6) There exists a $V(H)$-path $R$  in $J \setminus S$  with at least one end not in $\{a,b\}$ connecting distinct components of $H\setminus S$.
\medskip

Suppose every $V(H)$-path $Q$ in $J \setminus S$   connecting distinct components of $H\setminus S$ joins $a$ and $b$.
Then, 
each of $\Gamma_1$, $\Gamma_2$, and $\Gamma_3$ contains  exactly one vertex in $S$.
If $Q$ has length at least $2$, 
then
using $Q$ instead of the path $\Gamma_i$ containing a vertex of $S$, 
we obtain a subgraph $H'$ isomorphic to a subdivision of $K_{2,3}$ with $|V(H')\cap S| < |V(H)\cap S|$, contradicting the choice of $H$.
Hence, the length-one path $ab$ is a unique $V(H)$-path joining two distinct components of $H\setminus S$.
Let $H''$ be the subgraph of $J$ obtained from $H$ by adding the edge $ab$.
Then $(H'',V(H'') \cap S)$ is a $K_{2,3}$-snare and so $J=H''$.
If $J\setminus S$ contains a degree-$2$ vertex $x$, then $(J/x,S)$ is a $K_{2,3}$-snare, a contradiction, so, $\Gamma_i$ has length $2$ for $i=1,2,3$.
Hence, $J$ is isomorphic to $K^{+}_{2,3}$, in contradiction to our assumption.
This proves (6).
\bigskip

We may assume that path $R$ satisfying (6) joins
$x\in V(\Gamma_1)- S$ and $y \in V(\Gamma_2)- S$ where $\{x,y \} \neq \{a,b\}$.
By (2), it follows that $x,y \notin \{a,b\}$.

Since $x$ and $y$ are contained in distinct components of $H\setminus S$, 
$\Gamma_1$ or $\Gamma_2$ contains a vertex in $S$.
We may assume that $\Gamma_1$ contains at least one vertex in $S$.
Then, by (5),
we know  that $|V(\Gamma_1) \cap S| \ge 2$, %
$\Gamma_2$ has length two, and $y$ is the internal vertex of $\Gamma_2$. 
This also means $a$ and $b$ belong to the same component of $H\setminus S$. 

Let $\mathcal{Q}$ be a minimal set of $V(H)$-paths in $J\setminus S$ 
such that $R\in \mathcal{Q}$ and the union of $H\setminus S$ and all paths in $\mathcal{Q}$ is connected. 

Suppose $\mathcal{Q}$ contains a $V(H)$-path $Q$  from $w$ to $z$ where $w\in V(\Gamma_1)$ and $z\in V(\Gamma_3)$. 
By (5), since $\Gamma_1$ contains a vertex in $S$, $\Gamma_3$ must have length two and have no vertices in $S$.
Since $a,b,y,z$ are contained in the same component of $H\setminus S$, 
$x$ and $w$ belong to distinct components of $H\setminus S$
by the minimality of $\mathcal{Q}$.
We may assume that $x\in I(a,w)$. %
Since $R, Q \in \mathcal{Q}$, by the definition of $\mathcal{Q}$, there exist $s_1,s_2,s_3 \in V(\Gamma_1)\cap S$ such that $s_1 \in I(a,x)$, $s_2\in I(x,w)$ and $s_3 \in I(w,b)$.
Then, $R$ and $Q$ are internally disjoint by (2), and 
so there are three internally disjoint paths in the union of $H$, $R$, and $Q$ of length at least $2$ joining $x$ and $z$ not using any vertex in $I(w,b)$. 
They form a subdivision of $K_{2,3}$ which contradicts the choice of $H$.
Therefore, there is no $V(H)$-path in $\mathcal{Q}$ joining $\Gamma_1$ and $\Gamma_3$, 
which implies that every $V(H)$-path in $\mathcal{Q}$ joins $V(\Gamma_1)- \{a,b\}$ and $V(\Gamma_2)- \{a,b\}$, or $V(\Gamma_3)- \{a,b\}$ and $V(\Gamma_2)-\{a,b\}$ by (2).
Since $\{y\}=V(\Gamma_2)- \{a,b\}$, we conclude that every $V(H)$-path in $\mathcal{Q}$ contains $y$ as an end vertex.
\bigskip

\noindent(7) All $V(H)$-paths in $\mathcal{Q}$ are internally disjoint.
\medskip

Suppose $Q_1,Q_2 \in\mathcal{Q}$ have a common internal vertex.
Let $x_i$ be the end vertex of $Q_i$ other than $y$ for $i=1,2$, and let $w$ be the intersection of $Q_1$ and $Q_2$ closest to $x_2$ in $Q_2$.
If $x_1$ and $x_2$ are in the same path $\Gamma_i$ for some $i=1,3$, 
then there must be a vertex in $I(x_1,x_2)\cap S$ by the minimality of $\mathcal{Q}$, and it leads to a contradiction by (2).
Hence,  we may assume that $x_1 \in V(\Gamma_1)$ and $x_2 \in V(\Gamma_3)$.
Then, there are three internally disjoint paths from $a$ to $w$ in the union of $H$, $Q_1$ and $Q_2$ of length at least two, 
not using any vertex in $I(x_1,b) \cup I(x_2,b)$, which form  a subdivision $H'$ of $K_{2,3}$.
Since $I(x_1,b)\cap S \neq \emptyset$, we have $|V(H')\cap S| < |V(H)\cap S|$, which yields a contradiction to our assumption.
This proves (7). 
\bigskip

If $H$ contains exactly two vertices in $S$, 
then $H\setminus S$ has exactly two components, 
and so $\mathcal{Q}=\{R\}$, and $|I(a,x)\cap S|=|I(x,b)\cap S|=1$.
Then, $(J\setminus ay, S)$ is an $K_{2,3}$-snare, a contradiction to the definition of a $K_{2,3}$-trap. 
Therefore, $H$ contains at least three vertices in $S$. 
Let $C$ be the cycle consisting of $\Gamma_1$ and $\Gamma_3$, 
and let $s_1,s_2,\ldots, s_k$ be the vertices in $S\cap V(H)$ in the cyclic order. In particular, $k \geq 3$.
We may assume, by rotating if necessary, that $I(a,s_1)$ contains no vertices of $S$ and $s_1 \in I(a,s_2)$.

Since $S$ is independent, $C\setminus S$ consists of exactly $k$ paths.
We define %
paths $Q_1,Q_2,\ldots,Q_k$ in $G$ from $y$ to each path of $C\setminus S$ as follows:
\begin{itemize}

	\item For $i=1,2,\ldots,k-1$, 
	if $s_{i}$ and $s_{i+1}$ are in $\Gamma_j$ for some $j=1,3$, then let $Q_i$ be the $V(H)$-path in $\mathcal{Q}$ connecting $I(s_i,s_{i+1})$ and $y$, 
	and if $s_i \in V(\Gamma_1)$ and $s_{i+1} \in V(\Gamma_3)$, then let $Q_i$ be the length-one path $yb$.
	\item Let $Q_k$ be the length-one path $ay$.  
\end{itemize}

By (7), $Q_1,Q_2,\ldots,Q_k$ are internally disjoint, and by the definition of $\mathcal{Q}$, they have no vertices in $S$.
Let $H'=C \cup Q_1 \cup \ldots \cup Q_k$. Then $(H',V(C) \cap S)$ is a $K_{2,3}$-snare, and so $J = H'$. By minimality of $J$, each $Q_i$ has length one, and the end of $Q_i$ is the unique vertex on $C$ between $s_i$ and $s_{i+1}$ in the cyclic order where $s_{k+1}:=s_1$. It follows that $J$ is isomorphic to $W_k^{+}$ and $S$ is the set of degree-$2$ vertices of $J$. This contradiction finishes the proof.
\end{proof}

Now we prove  Theorem~\ref{thm:outer_vertex}.
It is well known that a graph is outerplanar if and only if it has no topological minor
isomorphic to $K_4$ or $K_{2,3}$ \cite{CH67}.

\outervertex*
\begin{proof}%
  Let $\mc{F}=\{K_4,K_{2,3}\}$. By Theorem~\ref{t:vbrittle}, $\kappa_{\mc{O}}$ is bounded on an ideal $\mc{G}$ if and only if
$\Fan(J,S) \not \subseteq \mc{G}$ whenever $(J,S)$ is a $K_4$-trap or a $K_{2,3}$-trap.
Thus it suffices to show that for every such $(J,S)$ the ideal $\Fan(J,S)$ contains one of the ideals listed in the statement of  Theorem~\ref{thm:outer_vertex}. If $(J,S)$ is a $K_{2,3}$-trap then this follows from Lemma~\ref{l:outer_trap}. 
	
If $(J,S)$ is a $K_{4}$-trap then either $J=K_4$, in which case $|S| \leq 1$ and $\Fan(J,S)$ is one of the ideals in the statement of  Theorem~\ref{thm:outer_vertex},   or $J$ contains a topological minor isomorphic to $K_{2,3}$. In the last case, there exists a $K_{2,3}$-trap $(J',S')$ such that $\Fan(J',S') \subseteq \Fan(J,S)$ and the desired conclusion holds as we already established it for $K_{2,3}$-traps.
\end{proof}

\section{Graphs with large $\mc{O}$-edge-brittleness and small edit distance from $\mc{O}$}\label{sec:example}

In this section we prove Proposition~\ref{prop:example}.

A \emph{hemmed graph} is a pair $(G,P)$, where $P$ is a path in $G$. Given a hemmed graph $(G,P)$, with $V(P)=\{v_1,v_2,\ldots,v_k\}$, indexed in the order of appearance on $P$, 
$\sigma(G,P)$ is defined to be a hemmed graph $(G',P')$ 
where $V(G')=V(G)\cup \{u_1,\ldots,u_{k-1}\}$ for new vertices $u_1,\ldots,u_{k-1}$ 
and 
$P'$ is a path with vertex set $\{v_1,u_1,v_2,u_2,\ldots,u_{k-1},v_k\}$ in order.

\begin{LEM}\label{l:etaplus} Let $G$ be a graph, $\C$ be a monotone graph class, and $P$ be a path in $G$ such that $G[V(P)] \not \in \C$. Let $(G',P')=\sigma(G,P)$. Then $\eta_{\C}(G') \geq \eta_{\C}(G)+1$.
\end{LEM}
\begin{proof}
		Let $\mc{V}'$ be a partition of $V(G')$ such that $G[V]\in \C$ for all $V \in \mc{V}'$ and $\eta_{\C}(G')$ edges of $G'$ have ends in distinct parts of $\mc{V}'$. By the choice of $P$, the set $V(P)$ intersects at least two parts of $\mc{V}'$, and it follows that at least one edge of $P'$ has  ends in distinct parts of $\mc{V}'$. Let $\mc{V}$ be the restriction of $\mc{V}'$ to $V(G)$. By the observation above, at most $\eta_{\C}(G')-1$ edges of $G$ have ends in distinct parts of $\mc{V}'$, implying $\eta_{\C}(G) \leq \eta_{\C}(G')-1$, as desired.
\end{proof}

We say that a hemmed graph $(G,P)$ is \emph{outerplanar}, if $G$ admits an outerplanar drawing such that $P$ is a part of the boundary of the outer face. The following useful observation follows immediately from the definitions.

\begin{OBS}\label{o:hemmed}  Let $(G,P)$ be an outerplanar hemmed graph. Then $\sigma(G,P)$ is outerplanar.
\end{OBS}

We are now ready to prove Proposition~\ref{prop:example}.
\example*
\begin{proof}%
	Let $G_1=K_4$ and let $P_1$ be any path in $G_1$ with $V(P_1)=V(G_1)$. Define the hemmed graph  $(G_{\ell},P_\ell)$ for $\ell \ge 2$ recursively, as $(G_\ell,P_\ell)=\sigma(G_{\ell-1},P_{\ell-1})$. 
	
	We will show that $e_{\mc{O}}(G_\ell)=1$ and $\eta_{\mc{O}}(G_\ell) \geq \ell+1$, implying the proposition.
	
	Let $f \in E(G_1) - E(P_1)$ be the edge of $G_1$ joining one of the ends of $P_1$ to an internal vertex of $P_1$. Then $(G_1 \setminus f, P_1)$ is an outerplanar hemmed graph.
	Inductively, it follows from  Observation~\ref{o:hemmed} that $(G_\ell \setminus f, P_\ell)$ is outerplanar. In particular, $G_\ell \setminus f$ is outerplanar, implying $e_{\mc{O}}(G_\ell) = 1$.
	
Lemma~\ref{l:etaplus} implies that  $\eta_{\mc{O}}(G_\ell) \geq \eta_{\mc{O}}(G_{\ell-1})+1$ for every $\ell \geq 2$. Therefore, since $\eta_{\mc{O}}(G_1)\ge 2$, we have $\eta_{\mc{O}}(G_\ell) \geq \ell+1$, as claimed.
\end{proof}

\section{Concluding remarks.}\label{sec:remarks}

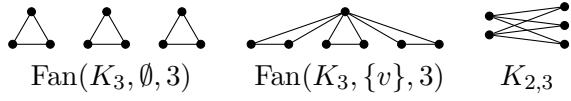
\begin{figure}[t]
  \centering
  \tikzstyle{v}=[circle, draw, solid, fill=black, inner sep=0pt, minimum width=3pt]
  \begin{tabular}{ccc}
 \begin{tikzpicture}
   \foreach\x in {0,1,2}{
     \begin{scope}[xshift=\x cm]
       \draw(0,0) node[v]{}--(.5,0) node[v]{}-- (60:.5) node[v]{}--(0,0);
     \end{scope}
     }
   \end{tikzpicture}
   \qquad&
 \begin{tikzpicture}
     \node[v] at (60:.5) (t) {};
   \foreach\x in {-1,0,1}{
     \begin{scope}[xshift=\x cm]
       \draw(0,0) node[v]{}--(.5,0) node[v]{}-- (t)--(0,0);
     \end{scope}
     }
   \end{tikzpicture}
   \qquad&
   \begin{tikzpicture}
     \node[v] at (0,.375) (x){};
     \node[v] at (0,.125) (y){};
     \foreach \y in {0,0.25,.5}{
       \draw (x)--(1,\y) node[v]{} -- (y);
       }
     \end{tikzpicture}
    \\
    $\fan(K_{3},\emptyset,3)$
         & $\fan(K_{3},\{v\},3)$
         & $K_{2,3}$
    \end{tabular}
  \caption{Obstructions for partitioning into forests.}
  \label{fig:eta_a}
\end{figure}

\begin{figure}[t]
  \centering
  \tikzstyle{v}=[circle, draw, solid, fill=black, inner sep=0pt, minimum width=3pt]
  \begin{tabular}{cccc}
 \begin{tikzpicture}
   \foreach\x in {0,1,2}{
     \begin{scope}[xshift=\x cm]
       \draw(0,0) node[v]{}--(.5,0) node[v]{}-- (60:.5) node[v]{}--(0,0)--
       (-60:.5)node[v]{}--(.5,0);
     \end{scope}
     }
   \end{tikzpicture}
      &   \begin{tikzpicture}
      \node [v] at (60:.5) (t){};
      \foreach\x in {-1,0,1}{
      	\begin{scope}[xshift=\x cm]
      	\draw(0,0) node[v]{}--(.5,0) node[v]{}-- (t)--(0,0)--
      	(-60:.5)node[v]{}--(.5,0);
      	\end{scope}
      }
      \end{tikzpicture}&
  \begin{tikzpicture}
   \node[v] at (0,0) (t) {};
   \foreach \x in {20,90,160}{
     \draw(t)-- (\x:1) node[v]  (b\x){};
     \draw(t)-- (\x+20:1) node[v] (a\x){};
     \draw(t)-- (\x-20:1) node[v](c\x){};
     \draw (a\x)--(b\x)--(c\x);
     }
   \end{tikzpicture}
&
   \begin{tikzpicture}[yscale=1.5]
     \node[v] at (0,.375) (x){};
     \node[v] at (0,.125) (y){};
     \foreach \y in {0,0.25,.5}{
       \draw (x)--(1,\y) node[v]{} -- (y);
     }
     \end{tikzpicture}
   \\
   $\fan(D,\emptyset,3)$
   &$\fan(D,\{u\},3)$ 
   & $\fan(D,\{v\},3)$
   & $K_{2,3}$
  
     \end{tabular}
  \caption{Obstructions for partitioning into diamond-free graphs.}
  \label{fig:eta_d}
\end{figure}
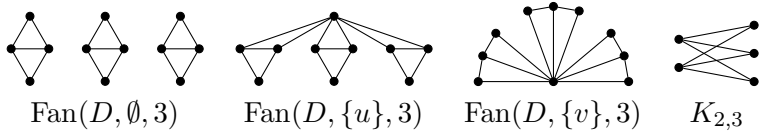

\begin{figure}[t]
  \centering
  
  \tikzstyle{v}=[circle, draw, solid, fill=black, inner sep=0pt, minimum width=3pt]
  \begin{tabular}{cccc}
      \begin{tikzpicture}
        \foreach \y in {0,...,4}{
          \foreach \i in {0,...,3} {
            \draw (\i * 1/3,0.4*\y  - .2) node[v](v\i) {};
          }
          \draw (v0)--(v3);
          \draw (v0) [out=-20,in=200] to (v2);
          \draw (v1) [out=-20,in=200] to (v3);
          \draw (v0) [out=-20,in=200] to (v3);
        }
      \end{tikzpicture}
    &
      \begin{tikzpicture} %
        \draw (0,1.1) node[v] (v) {};
        \foreach \i in {0,1,2,3}{
          \foreach \j in {0,2}{
            \draw (1,.6*\i+.2*\j) node[v](v\i-\j) {};
          }
          \draw (.8,.6*\i+.2) node[v](v\i-1) {};
          \draw (v\i-0)--(v\i-1)--(v\i-2)--(v\i-0);
          \foreach \j in {0,1,2}{
            \draw (v) [out=\i*60+\j*15-105,in=180]to (v\i-\j);
          }
        }
      \end{tikzpicture}
    &
      \begin{tikzpicture}
        \foreach \y in {0,...,4}{
          \foreach \i in {0,...,2} {
            \draw (\i * 1/3,0.4*\y - .2) node[v](v\i) {};
          }
          \foreach \i in {0,1} {
            \draw (.3*\i+.2,0.4*\y ) node[v](w\i) {};
          }
          \foreach \i in {0,1}{
            \foreach \j in {0,...,2}{
              \draw (w\i)--(v\j);
            }
          }
        }
      \end{tikzpicture}
    &
      \begin{tikzpicture} 
        \draw (0,1) node[v] (v){};
        \foreach \j in {0,...,3}{
          \draw (.85,0.6*\j) node[v] (w\j){};
          \draw (.85,0.6*\j+0.25)node[v](u\j){};
          \draw (.7,0.6*\j+0.125)node[v](x\j){};
          \draw (1,0.6*\j+0.125)node[v](y\j){};
          \draw (w\j)--(x\j)--(u\j)--(y\j)--(w\j);
        }
        \foreach \j in {0,...,3}{
          \draw (v) [out=\j*40-70,in=180] to (w\j);
          \draw (v)   [out=\j*40-50,in=180] to (u\j);
        }
      \end{tikzpicture}
        \\
        $\fan(K_4,\emptyset,5)$
        & $\fan(K_{4},\{v\},4)$
        & $\fan(K_{2,3},\emptyset,5)$
	&$\fan(K_{2,3},\{u\},4)$
  \end{tabular}
  \\\medskip
  \begin{tabular}{cccc}
      \begin{tikzpicture} 
        \draw (0,1) node[v] (v){};
        \foreach \j in {0,...,3}{
          \draw (.8,0.6*\j) node[v] (w\j){};
          \draw (.8,0.6*\j+0.3)node[v](u\j){};
          \draw (.8,0.6*\j+0.15)node[v](x\j){};
          \draw (1,0.6*\j+0.15)node[v](y\j){};
          \draw (w\j)--(y\j);
          \draw (u\j)--(y\j);
          \draw (x\j)--(y\j);
        }
        \foreach \j in {0,...,3}{
          \draw (v) [out=\j*40-70,in=180] to (w\j);
          \draw (v) [out=\j*40-60,in=180] to (x\j);
          \draw (v)   [out=\j*40-50,in=180] to (u\j);
        }
      \end{tikzpicture}
    &
      \begin{tikzpicture} 
        \draw (0,.9) node[v] (v){};
        \draw (0,1.2) node[v] (w){};
        \foreach \j in {0,...,3}{
          \draw (1,0.6*\j) node[v] (w\j){};
          \draw (1,0.6*\j+0.3)node[v](u\j){};
          \draw (1,0.6*\j+0.15)node[v](x\j){};
          \draw (w\j)--(x\j)--(u\j);
        }
        \foreach \j in {0,...,3}{
          \draw (v)-- (w\j)--(w);
          \draw (v)-- (u\j)--(w);
        }
      \end{tikzpicture}
&
      \begin{tikzpicture} %
        \foreach \i in {0,1,2}{
          \draw (0,.5*\i+.5) node[v] (v\i){};
        }
        \foreach \j in {0,...,3}{
          \draw (1,0.6*\j) node[v] (w\j){};
          \draw (1,0.6*\j+0.25)node[v](u\j){};
          \draw (w\j)--(u\j);
        }
        \foreach \i in {0,1,2}{
          \foreach \j in {0,...,3}{
            \draw (u\j)--(v\i)--(w\j);
          }
        }
      \end{tikzpicture}
    &
      \begin{tikzpicture}[scale=.7]
        \foreach \j in {0,1,2}{
          \draw (0.1-0.2*\j,0.5-0.3*\j) node[v] (x\j){};
        }
        \foreach \i in {0,3,4}{
          \draw (\i*72:1.4) node[v] (v\i){};
          \foreach \j in {0,1,2}{
            \draw  (\i*72+46-\j*10:.15*\j+1) node[v] (v\i-\j){};
            \draw (x\j)--(v\i-\j);
            \draw plot [smooth] coordinates { (v\i) (v\i-\j) (\i*72+72:1.4)};
          }
        }
        \foreach \i in {1,2}{
          \draw (\i*72:1.4) node[v] (v\i){};
          \foreach \j in {0,1,2}{
            \draw  (\i*72+26+\j*10:.15*\j+1) node[v] (v\i-\j){};
            \draw (x\j)--(v\i-\j);
            \draw plot [smooth] coordinates { (v\i) (v\i-\j) (\i*72+72:1.4)};
          }
        }
      \end{tikzpicture}
    \\

	 $\fan(K_{2,3},\{w\},4)$
        & $\fan(K_{2,3},\{u,u'\},4)$
    & $\fan(K^{+}_{2,3},S',4)$        
& $\fan(W^{+}_{5},S_5^+,3)$
  \end{tabular}
  \caption{Obstructions for partitioning into outerplanar graphs.}
  \label{fig:kappa_o}
\end{figure}
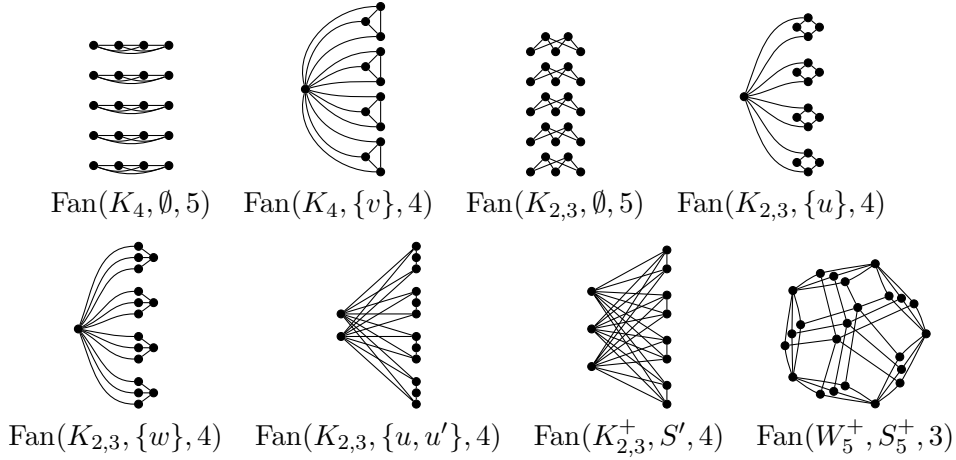

One can interpret our results as Ramsey-type results regarding the edit distance,  the $\bar{\C}$-capacity, the edge-brittleness, and the vertex-brittleness
as described in the following corollaries immediately implied by Theorems \ref{thm:forest_edge}, \ref{thm:diamond_edge}, \ref{thm:forest_vertex}, and \ref{thm:outer_vertex}.

\begin{COR}\label{cor:forest_edge}
Let $n$ be a positive integer.
If a graph has sufficiently large $\nu_{\mc A}$, $e_{\mc A}$, or $\eta_{\mc A}$, then it contains a topological minor isomorphic to $\fan(K_3,\emptyset, n)$, $\fan(K_3, \{v\}, n)$, or $K_{2,n}$
for a vertex $v\in V(K_3)$. See Figure~\ref{fig:eta_a}.
\end{COR}

\begin{COR}\label{cor:diamond_edge}
Let $n$ be a positive integer.
If a graph has sufficiently large  $\nu_{\mc D}$, $e_{\mc D}$, or $\eta_{\mc D}$, then it contains a topological minor isomorphic to $\fan(D,\emptyset, n)$, $\fan(D,\{u\},n)$, $\fan(D,\{v\},n)$, or $K_{2,n}$ where  $u$ is a vertex of degree $2$ and $v$ is a vertex of degree $3$ in $D$. See Figure~\ref{fig:eta_d}.
\end{COR}

\begin{COR}\label{cor:forest_vertex}
Let $n$ be a positive integer.
 If a graph has sufficiently large $\kappa_{\mc A}$, then it has a topological minor isomorphic to $\fan(K_3,\emptyset,n)$ or $\fan(K_3,\{v\},n)$ where $v\in V(K_3)$.
\end{COR}

\begin{COR}\label{cor:outer_vertex}
Let $n$ be a positive integer.
If a graph has sufficiently large $\kappa_{\mc O}$, then it has a topological minor isomorphic to one of the following. See Figure~\ref{fig:kappa_o}.
\begin{itemize}
\item $\fan(K_4,\emptyset, n)$, $\fan(K_4, \{v\},n)$ where $v\in V(K_4)$.
\item $\fan(K_{2,3},\emptyset,n)$, $\fan(K_{2,3},\{u\},n)$, $\fan(K_{2,3},\{w\},n)$, $\fan(K_{2,3},\{u,u'\},n)$ where $u$ and $u'$ are degree-$2$ vertices of $K_{2,3}$ 
and $w$ is a degree-$3$ vertex of $K_{2,3}$.
\item $\fan(K_{2,3}^+,S', n)$ where $S'$ is the set of all degree-$2$ vertices of $K_{2,3}^+$.

\item $\fan(W_k^+,S_k^+, n)$ for some $k\ge 3$ where $S_k^+$ is  the set of degree-$2$ vertices of $W_k^+$.
 \end{itemize}
\end{COR}

We investigated qualitative relationship between several measures ($e_{\C}$, $\eta_{\C}$, $\nu_{\C}$, and $\kappa_{\C}$) of distance from an ideal $\C$. Theorem~\ref{thm:edge} shows that if $D\notin \C $ and $\C$ is determined by excluding a finite number of $2$-connected graphs as topological minors, then $e_{\C}$, $\eta_{\C}$, and $\nu_{\C}$  are tied to each other.\footnote{That is, each of them is bounded by a function of the others.} The last  condition can likely be relaxed at the expense of a more technical argument, but the condition $D\notin\C $ is crucial as Proposition~\ref{prop:example} shows. It would be interesting to determine the exact threshold at which the change of behavior occurs.

\begin{QUE}
	What are the minimal minor-closed classes $\C$ such that $\eta_{\C}$ is not bounded by a function of $e_{\C}$? 
\end{QUE}

Similarly, it would be interesting to determine the threshold beyond which Theorem~\ref{t:vbrittle} fails. 

Note that if a class $\C$ satisfies the conclusion of Theorem~\ref{t:vbrittle}, then in particular, $\kappa_{\C}$ is bounded by a function of $\nu_{\C}$. This does not hold for general ideals. Indeed, let $\mc{P}$ be the ideal of planar graphs, and let $\mc{Q}$ be the ideal of graphs with maximum degree at most $3$ embeddable in the projective plane. As the disjoint union of any two non-planar graphs has no projective planar embedding, it follows that $\nu_{\mc{P}}(G) \leq 1$ for every $G \in \mc{Q}$. On the other hand, it is easy to see that $\kappa_{\mc{P}}$ is unbounded on $\mc{Q}$ (and thus $e_{\mc{P}}$ and $\eta_{\mc{P}}$ are also unbounded by Observation~\ref{obs:basic}). 

\begin{QUE}
	What are the minimal minor-closed classes $\C$ such that $\kappa_{\C}$ (respectively,  $e_{\C}$) is not bounded by a function of $\nu_{\C}$? 
\end{QUE}  

\noindent {\bf Acknowledgement.} This research was partially completed at the \emph{2019 Barbados Graph Theory Workshop} held at the Bellairs Research Institute. We thank the participants of the workshop for providing a stimulating environment for research.


\begin{thebibliography}{1}

  \bibitem{AlonStav08}
  N.~Alon and U.~Stav.
  \newblock The maximum edit distance from hereditary graph properties.
  \newblock {\em J. Combin. Theory Ser. B}, 98(4):672--697, 2008.
  
  \bibitem{CH67}
  G.~Chartrand and F.~Harary.
  \newblock Planar permutation graphs.
  \newblock {\em Ann. Inst. H. Poincar\'{e} Sect. B (N.S.)}, 3:433--438, 1967.
  
  \bibitem{CHS1979}
  E.~J. Cockayne, S.~T. Hedetniemi, and P.~J. Slater.
  \newblock Matchings and transversals in hypergraphs, domination and
    independence in trees.
  \newblock {\em J. Combin. Theory Ser. B}, 26(1):78--80, 1979.
  
  \bibitem{GL1969}
  A.~Gy\'{a}rf\'{a}s and J.~Lehel.
  \newblock A {H}elly-type problem in trees.
  \newblock In {\em Combinatorial theory and its applications, {II} ({P}roc.
    {C}olloq., {B}alatonf\"{u}red, 1969)}, pages 571--584. North-Holland,
    Amsterdam, 1970.
  
  \bibitem{Martin16}
  R.~R. Martin.
  \newblock The edit distance in graphs: methods, results, and generalizations.
  \newblock In {\em Recent trends in combinatorics}, volume 159 of {\em IMA Vol.
    Math. Appl.}, pages 31--62. Springer, [Cham], 2016.
  
  \bibitem{M1927}
  K.~Menger.
  \newblock Zur allgemeiner {K}urventheories.
  \newblock {\em Fund. Math}, 16(1):96--115, 1927.
  
  \bibitem{RS1986}
  N.~Robertson and P.~Seymour.
  \newblock Graph minors. {V}. {E}xcluding a planar graph.
  \newblock {\em J. Combin. Theory Ser. B}, 41(1):92--114, 1986.
  
  \end{thebibliography}
\end{document}